\documentclass{article}
\usepackage[T1]{fontenc}
\usepackage[english]{babel}
\usepackage{graphicx}
\usepackage{slashed}
\usepackage{stmaryrd}
\usepackage{amsmath,amsfonts,amssymb}

\usepackage{ntheorem}
\usepackage{dsfont}
\usepackage{hyperref}
\usepackage{tikz}
\usepackage{pgfplots}
\usetikzlibrary{patterns}
\usetikzlibrary{positioning,arrows,arrows.meta}
\usepackage{geometry}
\usepackage{enumitem}
\theoremseparator{.} 
\newtheorem{Th}{Theorem}[section]
\newtheorem{Def}[Th]{Definition}
\newtheorem{Rq}[Th]{Remark}
\newtheorem{Pro}[Th]{Proposition}
\newtheorem{Cor}[Th]{Corollary}
\newtheorem{Lem}[Th]{Lemma}
\newcommand{\R}{\mathbb{R}}

\newcommand{\T}{\mathbf{T}}
\newcommand{\E}{\mathbb{E}}

\newcommand{\VV}{\widehat{\mathbb{V}}}
\newcommand{\ZZ}{\widehat{Z}}
\newcommand{\I}{\mathfrak{I}}
\newcommand{\II}{\mathcal{I}}

\newcommand{\Vv}{\mathbb{V}}
\newcommand{\Sp}{\mathbb{S}}
\newcommand{\V}{\mathbf{k}}

\newenvironment{proof}{\noindent\textit{Proof.~}}{\hfill$\blacksquare$\bigbreak} 
\newlength{\plarg}
\title{A vector field method for massless relativistic transport equations and applications}

\author{Léo Bigorgne\footnote{Laboratoire de Mathématiques, Univ. Paris-Sud, CNRS, Université Paris-Saclay, 91405 Orsay.}}

\begin{document}

\maketitle
    
\begin{abstract}
In this article, we present a vector field method for the study of solutions to massless relativistic transport equations. Compared to \cite{FJS}, we remove the Lorentz boosts of the commutation vector fields and we prove a new functional inequality in order to derive pointwise decay estimates on the solutions. It makes our method more suitable for the study of systems coupling a massless Vlasov equation with an equation which has a different speed of propagation. We also believe that our approach can be adapted more easily to the context of curved backgrounds. 

In the second part of this paper, we apply our method in order to derive sharp decay estimates on the spherically symmetric small data solutions of the relativistic massless Vlasov-Poisson system. No compact support assumption is required on the data and, in particular, the initial decay in the velocity variable is optimal. In order to compensate the slow decay rate of the particle density near the light cone, we take advantage of the null structure of the equations.
\end{abstract}

\section{Introduction}

The purpose of this article is to provide a methodology for the study of massless relativistic transport equations which does not rely on the full Lorentz invariance of the operator\footnote{In this article, we use the Einstein summation convation. A sum on Greek indices goes from $0$ to $3$ whereas a sum on Latin indices goes from $1$ to $3$. For instance, $v^i \partial_i= \sum_{i=1}^3 v^i \partial_i$.} $ \T:=|v| \partial_t + v^i \partial_i$. More precisely, our goal is to present a modification of the vector field method developed by Fajman-Joudioux-Smulevici in \cite{FJS}, which consists in 
\begin{enumerate}
\item Commuting the equations by the vector fields of the set $\widehat{\mathbb{K}}_0$, which are
\begin{itemize}
\item $\partial_t$, the translation in time, and $\partial_k$, for $1 \leq k \leq 3$, the translations in space.
\item $S:= t \partial_t+r \partial_r$, the scaling vector field in $(t,x)$.
\item $S_v:= v^i \partial_{v^i}$, the scaling vector field in $v$.
\item $\widehat{\Omega}_{ij}:= x^i \partial_j-x^j \partial_i+v^i \partial_{v^j}-v^j \partial_{v^i}$, for $1 \leq i < j \leq 3$, the complete lift\footnote{We refer to \cite{FJS} for an introduction to the complete lift of a vector field $X^{\mu} \partial_{\mu}$ (see Section $2.7$ and Appendix $C$).} of the rotational vector field $\Omega_{ij} := x^i \partial_j-x^j \partial_i$.
\item $\widehat{\Omega}_{0k}:= x^k \partial_t+t \partial_k+|v| \partial_{v^k}$, for $1 \leq k \leq 3$, the complete lift of the Lorentz boost $\Omega_{0k} := x^k \partial_t+t\partial_k$.
\end{itemize}
\item Prove boundedness on $L^1$ norms of the form $\| w^p  \widehat{Z}^{\beta} f \|_{L^1_{x,v}}$, where $\widehat{Z}^{\beta} \in \widehat{\mathbb{K}}_0^{|\beta|}$ is a combination of $|\beta|$ vector fields of $\widehat{\mathbb{K}}_0$, $p \in \mathbb{N}$ and $w \in \mathbf{k}_0$ is a weight preserved by $\T$. The set $\mathbf{k}_0$ is composed of
\begin{itemize}
\item $\frac{v^{\mu}}{|v|}$, $0 \leq \mu \leq 3$.
\item $s:= t-x^i\frac{v_i}{|v|}$, the scaling.
\item $z_{ij}:=x^i\frac{v^j}{|v|}-x^j\frac{v^i}{|v|}$, for $1 \leq i < j \leq 3$, the rotations.
\item $z_{0k}:= x^k-t\frac{v^k}{|v|}$, for $1 \leq k \leq 3$, the Lorentz boosts.
\end{itemize}
\item Obtain pointwise decay estimates on $\int_{\R^3_v} |f| dv$ through Klainerman-Sobolev type inequalities, such as
\begin{equation}\label{th6}
 \forall \hspace{0.5mm} (t,x) \in [0,T[ \times \R^3, \hspace{0.8cm} \int_{\R^3_v} |g|(t,x,v) dv \hspace{2mm} \lesssim \hspace{2mm} \frac{1}{(1+t+r)^2(1+|t-r|)} \sum_{|\beta| \leq 3} \sum_{\widehat{Z}^{\beta} \in \widehat{\mathbb{K}}_0} \left\| \widehat{Z}^{\beta} f (t, \cdot, \cdot ) \right\|_{L^1_{x,v}}, 
\end{equation}
 which strongly relies on the use of the Lorentz boosts (see Theorem $6$ of \cite{FJS}).
\end{enumerate}

The vector field method was introduced by Klainerman \cite{Kl85} in the context of wave equations. In order to apply it to the study of coupled wave equations with different speed of propagation, Klainerman-Sideris adapted it in \cite{KlSid} to a smaller set of vector fields
$$\Vv \hspace{2mm} := \hspace{2mm} \left\{ \partial_t, \hspace{1mm} \partial_1, \hspace{1mm} \partial_2, \hspace{1mm} \partial_3, \hspace{1mm} S, \hspace{1mm} \Omega_{12}, \hspace{1mm} \Omega_{13}, \hspace{1mm} \Omega_{23} \right\} ,$$ 
which do not contain the Lorentz boosts $\Omega_{0k} := t \partial_k+x^k \partial_t$. Recently, Fajman-Joudioux-Smulevici extended this method in \cite{FJS} to massless and massive relativistic transport equations. This leads in particular to the proof of the stability of Minkowski spacetime as a solution to the massive Einstein-Vlasov system \cite{FJS3} (see also \cite{Lindblad}). The aim of this paper is, as \cite{KlSid} for the wave equation, to reduce the set of the commutation vector fields used in \cite{FJS} for the study of massless relativistic transport equations. More precisely, we will not use the complete lifts of the Lorentz boost vector fields $\widehat{\Omega}_{0k} =t \partial_k+x^k \partial_t+|v| \partial_{v^k}$ and we will commute the transport equations by
$$ \VV \hspace{2mm} := \hspace{2mm} \left\{ \partial_t, \hspace{1mm} \partial_1, \hspace{1mm} \partial_2, \hspace{1mm} \partial_3, \hspace{1mm} S, \hspace{1mm} S_v, \hspace{1mm} \widehat{\Omega}_{12}, \hspace{1mm} \widehat{\Omega}_{13}, \hspace{1mm} \widehat{\Omega}_{23} \right\} \hspace{2mm} = \hspace{2mm} \widehat{\mathbb{K}}_0 \setminus \left\{ \widehat{\Omega}_{01}, \hspace{1mm} \widehat{\Omega}_{02} , \hspace{1mm} \widehat{\Omega}_{03} \right\} .$$
As a consequence, we cannot use the Klainerman-Sobolev inequality \eqref{th6} anymore in order to derive sharp asymptotics on the solutions as it requires a good control of $\| \widehat{\Omega}_{0k} f\|_{L^1_{x,v}}$. Instead, we prove and use the following functional inequality (see Proposition \ref{decay}), which holds for all sufficiently regular function $g$,
\begin{eqnarray}
\nonumber \hspace{-2mm} \int_{\R^3_v} |g|(t,x,v) dv & \lesssim & \frac{1}{(1+|x|)^2 (1+|t-|x||)} \Bigg( \sum_{|\beta| \leq 2} \sum_{\widehat{\Omega}^{\beta} \in \VV^{|\beta|}}  \int_{ \{ t \} \times \R^3_y} \int_{\R^3_v} r \left| \widehat{\Omega}^{\beta} \left( \T ( g) \right) \right| \frac{dv}{|v|} dy \\
& & \hspace{4cm} +\sum_{|\xi| \leq 3} \sum_{\widehat{Z}^{\xi} \in \VV^{|\xi|}} \int_{\{ t \} \times \R^3_y} \int_{\R^3_v} \left(1+\left|t-x^i \frac{v_i}{|v|} \right| \right) \left| \ZZ^{\xi} g \right| dv dy \Bigg). \label{introdecay}
\end{eqnarray}
This modification of the vector field method developed in \cite{FJS} is then adapted to the study of systems coupling a massless relativistic transport equation with a wave equation which has a different speed of propagation or with an elliptic equation. Moreover, our method is also more likely to be adapted to massless Vlasov equations on a curved background such as Schwarzschild spacetime, in the spirit of \cite{DR09} for the wave equation. In this perspective, let us mention that we only use the translations in space as commutators in order to deal with the domain $r \leq 1$ and that we could avoid the use of the weights\footnote{We choose to work with the weights $z_{0k}$, $1 \leq k \leq 3$, for simplicity and since they do not prevent us to use our method for systems composed of a relativistic transport equation and an elliptic equation. We refer to Remark \ref{Lorentzweight} below to see how recover the results of this article without using them.} $z_{0k} \in \mathbf{k}_0$. Note that \cite{ABJ} proved, using a vector field method, an integrated decay result for solutions to massless relativistic transport equations on slowly rotating Kerr spacetime.

Subsequently, we will apply our method in order to derive the asymptotic behavior of the spherically symmetric small data solutions of the relativistic massless Vlasov-Poisson system
\begin{eqnarray}
 |v| \partial_t f + v^j \partial_j f +\sigma |v| \nabla^i \phi \cdot \partial_{v^i} f & = & 0 \label{VP1} \\
 \Delta \phi & = & \int_{v \in \R^3_v} f dv, \label{VP2}
\end{eqnarray}
where
\begin{itemize}
\item $\sigma \in \{-1,1 \}$. If $\sigma =1$ we are in the repulsive case and if $\sigma=-1$, we are in the attractive case.
\item The particle density $f$ depends on $(t,x,v) \in \R_+ \times \R^3 \times (\R^3 \setminus \{ 0 \} )$. In view of its physical meaning, $f$ is usually supposed non negative but since its sign play no role in this article, we do not restrict its values to $\R_+$.
\item The potential $\phi$ depends on $(t,x) \in \R_+ \times \R^3$.
\end{itemize} 
In the attractive case $\sigma=-1$ and for spherically symmetric solutions, the classical Vlasov-Poisson system, for which the transport equation \eqref{VP1} is replaced by $$ \partial_t f + v^j \partial_j f - \nabla^i \phi \cdot \partial_{v^i} f \hspace{2mm} = \hspace{2mm} 0,$$
can be obtained in a weak field and low velocity limit of the physically relevant Einstein-Vlasov system \cite{Ren94}. No similar results are known for the massless and the massive\footnote{For the massive system, equation \eqref{VP1} is replaced by $ \sqrt{1+|v|^2}\partial_t f + v^j \partial_j f - \sqrt{1+|v|^2} \nabla^i \phi \cdot \partial_{v^i} f \hspace{2mm} = \hspace{2mm} 0$.} relativistic Vlasov-Poisson systems, although they are expected to be obtained as a weak field limit of the Einstein-Vlasov system. Let us mention that when this supposed approximation does not hold, i.e. for strong field, finite-time blow up phenomena are known \cite{GlScRVP}.

Our result can be stated as follows.
\begin{Th}\label{theorem}
Let $\epsilon >0$, $N \geq 12$ and $f_0 : \R^3_x \times \R^3_v \rightarrow \R$ be a spherically symmetric function satisfying
$$ \sum_{|\alpha|+|\beta| \leq N+3} \int_{\R^3_x} \int_{\R^3_v} (1+|x|)^{|\alpha|+5} (1+|v|)^{|\beta|} \left| \partial^{\alpha}_x \partial^{\beta}_v f_0 \right| dx ds \hspace{1mm} \leq \hspace{1mm} \epsilon \hspace{1cm} \text{and} \hspace{1cm} \Big( |v| \leq 2 \Rightarrow f_0( \cdot ,v)=0 \Big).$$
Then, there exists $\epsilon_0 >0$ depending only on $N$ such that, if $\epsilon \leq \epsilon_0$, then the unique classical solution $(f,\phi)$ of the relativistic massless Vlasov-Poisson system \eqref{VP1}-\eqref{VP2} exists globally in time and satisfies the following properties.
\begin{itemize}
\item The particle density $f$ vanishes for small velocities, i.e. $|v| \leq 1 \Rightarrow f( \cdot , \cdot ,v)=0$.
\item Energy bounds for $f$ :
$$ \forall \hspace{0.5mm} t \in \R_+, \hspace{1cm} \sum_{z \in \mathbf{k}_0} \sum_{| \beta| \leq N} \sum_{\widehat{Z}^{\beta} \in \VV^{|\beta|}} \int_{\R^3_x} \int_{\R^3_v} |z|^2 \left| \widehat{Z}^{\beta} f \right| (t,x,v) dv dx \hspace{2mm} \lesssim \hspace{2mm} \epsilon.$$
\item Pointwise decay estimates on the velocity averages of $f$ and its derivatives,
$$ \forall \hspace{0.5mm} (t,x) \in \R_+ \times \R^3, \hspace{2mm} |\beta| \leq N-3, \hspace{1cm} \int_{\R^3_v} \left| \widehat{Z}^{\beta} f \right|(t,x,v) dv \hspace{2mm} \lesssim \hspace{2mm} \frac{\epsilon}{(1+|t-|x||)^2(1+|x|)^2}.$$
\item Boundedness for the potential,
$$ \forall \hspace{0.5mm} t \in \R_+,  \hspace{1cm} \sum_{|\gamma| \leq N} \sum_{Z^{\gamma} \in \Vv^{|\gamma|}} \left\| \nabla_x Z^{\gamma} \phi \right\|_{L^2(\Sigma_t)} \hspace{2mm} \lesssim \hspace{2mm} \frac{\epsilon}{\sqrt{1+t}}$$
\item Pointwise decay estimates for the potential,
$$ \forall \hspace{0.5mm} (t,x) \in \R_+ \times \R^3, \hspace{2mm} |\gamma| \leq N-3, \hspace{1cm} |\nabla_x Z^{\gamma} \phi|(t,x) \hspace{2mm} \lesssim \hspace{2mm} \frac{\epsilon}{(1+t)(1+|x|)}.$$
\end{itemize}
\end{Th}
\begin{Rq}
A similar result holds if we merely assume that $f_0$ vanishes for all $|v| \leq R$, with $R >0$. In that case, $\epsilon_0$ would also depend on $R$.
\end{Rq}
\begin{Rq}
Better estimates hold on the potential in certain cases. More precisely, if $Z^{\gamma} \phi$ is chargeless\footnote{We refer to Subsection \ref{subseccharge} for the definition of the charge of $Z^{\gamma} \phi$.}, then $\left\| \nabla_x Z^{\gamma} \phi \right\|_{L^2(\Sigma_t)} \hspace{0.5mm} \lesssim \hspace{0.5mm} \epsilon (1+t)^{-1}$ and if $Z^{\gamma} \phi$ is spherically symmetric (and $|\gamma| \leq N-3$), then $|\nabla_x Z^{\gamma} \phi| \hspace{0.5mm} \lesssim \hspace{0.5mm} \epsilon (1+t+|x|)^{-2}$.
\end{Rq}
\begin{Rq}
We start with optimal decay in $v$ on the particle density in the sense that we merely require $f_0(x, \cdot )$ to be integrable in $v$, which is a necessary condition for the source term of the Poisson equation \eqref{VP2} to be well defined.
\end{Rq}
To our knowledge, no results are known on the asymptotic behavior of the solutions of the relativistic massless Vlasov-Poisson system. In \cite{LMR}, Lemou-Méhats-Raphaël proved the existence of a family of finite time blow-up self similar solutions to \eqref{VP1}-\eqref{VP2} as well as the stability of their blow-up dynamic under radially symmetric perturbations. Concerning the classical system, the first results in that direction were obtained by Bardos-Degond in \cite{Bardos} where they proved global existence as well as optimal decay rates for $\int_{\R^3_v} f dv$ and $\nabla \phi$. The optimal decay estimates for the derivatives were established later by \cite{HRV}. The study of the small data solutions of the relativistic massive Vlasov-Poisson system was initiated by Glassey-Schaeffer \cite{GlScRVP}. They proved global existence for non negative spherically symmetric solutions compactly supported in $v$ in the repulsive case ($\sigma=1$). In the attractive case ($\sigma=-1$), global existence is only guaranteed under a smallness assumption and they proved that the solutions blow-up if the conserved energy
\begin{equation}\label{eq:conserve}
\mathfrak{E}_0 \hspace{2mm} := \hspace{2mm} \int_{\R^3_x} \int_{\R^3_v} f \sqrt{m^2+|v|^2} dv dx+ \frac{\sigma}{2} \int_{\R^3_x} |\nabla_x \phi|^2 dx, 
\end{equation}
where $m >0$ is the mass of the particles, is initially negative. We point out that their blow-up result also holds for the massless Vlasov-Poisson system\footnote{One only has to follow their proof. We emphasize that $f(0, \cdot , \cdot)$ has to be non negative.} ($m=0$). Later, Kiessling and Tahvildar-Zadeh precised these results in \cite{Kiessling} by providing critical norms beyond which there exists spherically symmetric initial data which give rise to solutions which blow-up. Recently, Wang proved in \cite{Wang} global existence and sharp asymptotics on the small data solutions of the classical and relativistic Vlasov-Poisson system. Finally, let us also mention the anterior work \cite{Poisson} of Smulevici, which has been improved recently by Duan \cite{Xianglong}, concerning the classical system. The asymptotic behavior of the solutions is derived using vector field methods for both the particle density and the potential. We refer to \cite{Taylor}, \cite{FJS}, \cite{dim4} and \cite{massless} for similar results on other massless Vlasov systems.

\subsection*{Main difficulties}

Recall that the central identity of the proof of\footnote{We refer to \cite{massless} (see Proposition $3.6$) for a concise proof of \eqref{th6}.} \eqref{th6} is
$$ (t-r) \partial_r \hspace{1mm} = \hspace{1mm} \frac{t}{t+r} \frac{x^i}{r} \Omega_{0i}-\frac{r}{t+r} S, \hspace{0.7cm} \text{leading to} \hspace{0.7cm} (1+|t-|x||) \left| \partial_r \int_{\R^3_v} |f| dv \right| \hspace{1mm} \lesssim \hspace{1mm} \sum_{\widehat{Z} \in \widehat{\mathbb{K}}_0} \int_{\R^3_v}\left( |f|+ | \widehat{Z}^{\beta} f| \right) dv . $$
The key step for obtaining \eqref{introdecay} then consists in proving such an inequality without using any Lorentz boost. Instead we use the commutators of $\VV$, the weight $s$ and the vector field $\T$, which makes our estimate adapted to the study of massless relativistic transport equations (see Proposition \ref{recovertauminus}).
\begin{Rq}
Note however that we do not fully recover in \eqref{introdecay} the decay rate of \eqref{th6} for, say, $t \geq 2r$. This issue can be solved by considering stronger weighted norms than the ones of \eqref{introdecay} (see Proposition \ref{decay}).
\end{Rq}

In the second part of this article, we study the small data solutions to the massless relativistic Vlasov-Poisson system. In contrast with the classical system, the Vlasov equation and the Poisson equation are of different nature. More concretely, \cite{Poisson} commute the Laplace operator with the uniform motions $t \partial_k$ and the (classical) transport operator $\partial_t +v^i \partial_i$ with $t \partial_k+\partial_{v^k}$. The relativistic transport operator $\T$ do not commute with $t \partial_k+\partial_{v^k}$ but with the complete lift of the Lorentz boosts $\widehat{\Omega}_{0k}=t \partial_k+x^k \partial_t+|v| \partial_{v^k}$ which are much different from $t \partial_k$. For this reason, we study the system \eqref{VP1}-\eqref{VP2} by using a smaller set of commutation vector fields and we then crucially use our new decay estimate for massless Vlasov field. 

Another difficulty arises from the small velocities (recall that $v \mapsto f(t,x,v)$ is defined on $\R^3 \setminus \{ 0 \}$). Indeed, the characteristics $(X,V)$ of the Vlasov equation \eqref{VP1} satisfy
$$ \dot{X} \hspace{2mm} = \hspace{2mm} \frac{V}{|V|} \hspace{1cm} \text{and} \hspace{1cm} \dot{V}(t) \hspace{2mm} = \hspace{2mm} \nabla_x \phi(t,X(t)),$$
so that the velocity part $V$ can reach the value $v=0$ in finite speed. We encounter a similar problem for the study of the massless Vlasov-Maxwell system in \cite{dim4} and \cite{massless} and we proved that if the particle density $f$ does not initially vanish for small velocities, then the system do not admit a local classical solution (see Proposition $8.1$ of \cite{dim4}). To circumvent this problem, we will then suppose that the velocity support of the Vlasov field is initially bounded away from $0$ and an important step of the proof will consist in proving that this property is propagated in time.

Finally, a crucial point of the proof consists in dealing with the weak decay rate of the solutions. In order to improve an energy bound on, say, $\| S f \|_{L^1_{x,v}}$, where $S$ is the scaling vector field, we will be led to control
$$ I \hspace{2mm} := \hspace{2mm} \int_0^t \int_{\R^3_x} \int_{\R^3_v} \nabla_x S \phi \cdot \nabla_v f dv dx ds.$$
The problem is that, even for a smooth spherically symmetric solution $(g,\psi)$ to 
\begin{eqnarray}
 \T (g) \hspace{2mm} = \hspace{2mm} |v| \partial_t g+v^i \partial_i g & = & 0, \label{simplepb} \\
\Delta \psi & = & \int_{v \in \R^3} g dv, \label{simplepb2}
\end{eqnarray}
$\partial_v g$ essentially behaves as $r  |\nabla_x g |+\sum_{\widehat{Z} \in \VV} | \widehat{Z} g|$ and $\nabla_x S \psi$ decay as $(1+t+r)^{-2}$, which leads to
$$  \int_0^t \int_{\R^3_x} \int_{\R^3_v} \left| \nabla_x S \psi \cdot \nabla_v g \right| dv dx ds \hspace{2mm} \lesssim \hspace{2mm} \sum_{\widehat{Z} \in \VV} \int_0^t \frac{1}{1+s} \| \widehat{Z} g (s, \cdot , \cdot ) \|_{L^1_{x,v}} ds \hspace{2mm} \lesssim \hspace{2mm} \log (1+t) .$$
In our case, this logarithmical growth on $\| S f \|_{L^1_{x,v}}$ would give a slightly worse decay rate on $\nabla_x S \phi$ than the one on $\nabla_x S \psi$, leading to
$$ |I| \hspace{2mm} \lesssim \hspace{2mm} \sum_{\widehat{Z} \in \VV} \int_0^t \frac{\log(1+s)}{1+s} \| \widehat{Z} f (s, \cdot , \cdot ) \|_{L^1_{x,v}} ds \hspace{2mm} \lesssim \hspace{2mm} \log^2 (2+t) \sum_{\widehat{Z} \in \VV} \sup_{0 \leq s \leq t} \| \widehat{Z} f (s, \cdot , \cdot ) \|_{L^1_{x,v}} $$
and preventing us to close the energy estimates. To circumvent this difficulty, we take advantage of the null structure of the system (see Lemma \ref{reducing}) in order to get
$$ |I| \hspace{2mm} \lesssim \hspace{2mm} \sum_{ k < l} \sum_{\widehat{Z} \in \VV} \int_0^t \int_{\R^3_x} \int_{\R^3_v}  | \Omega_{kl} S \phi | | \widehat{Z} f | dv dx ds +\sum_{w \in \V_0} \sum_{\widehat{Z} \in \VV} \int_0^t \int_{\R^3_x} \int_{\R^3_v}  | \nabla_x S \phi | | w \widehat{Z} f | dv dx ds.$$
Since
\begin{itemize}
\item the second term on the right hand side is uniformly bounded in time provided that $\| w \widehat{Z} f \|_{L^1_{x,v}}$ does not grow too fast,
\item $\Omega_{kl} S \phi$ vanishes since $S \phi$ is spherically symmetric,
\end{itemize}
one can expect to prove boundedness for $\| S f \|_{L^1_{x,v}}$ and, more generally, $\| \widehat{Z}^{\beta} f \|_{L^1_{x,v}}$.
\begin{Rq}
It is the combination of three difficulties which explains why our proof does not work for non spherically symmetric solutions. In order to expose them, let us return to our simplified system \eqref{simplepb}-\eqref{simplepb2}. The first problem is related to the weak decay rate of $\int_{\R^3_v} g dv$ and its derivatives near the light cone, leading, through standard elliptic estimate, to
\begin{equation}\label{eq:intro0}
 \| \nabla_x \psi \|^2_{L^2(\R^3)}(t) \hspace{2mm} \lesssim \hspace{2mm}  \| q^{-1} \psi \|_{L^2(\R^3)}(t) \left\| q \int_{\R^3} g dv \right\|_{L^2(\R^3)}(t) ,
 \end{equation}
where $q$ is a well chosen function of $(t-|x|,t+|x|)$. One can then obtain\footnote{This estimate is optimal in the presence of a non zero total charge, as in that case $\nabla_x \phi$ cannot decay faster than $r^{-2}$.} $\| \nabla_x \psi \|_{L^2(\R^3)} \lesssim (1+t)^{- \frac{1}{2}}$ by applying a Hardy type inequality such as the one of Appendix $B$ of \cite{LR}  and by taking advantage of the strong decay rate in $t-|x|$ of the Vlasov field. This would lead, using a Klainerman-Sobolev type inequality, to the following pointwise decay estimate on $\psi$,
$$ | \nabla_x \psi |(t,x) \hspace{2mm}\lesssim \hspace{2mm} \frac{1}{(1+|x|)\sqrt{1+t}},$$
which is much worse than the one obtained by our method in the spherically symmetric case. A second difficulty arises from the terms such as $\Omega_{kl} S \psi$, which do not vanish anymore. 

One way to deal with these issues could be 
\begin{itemize}
\item to use that $(1+|t-r|)^{-1}\Omega_{kl} S \psi$ behaves better than $r \nabla_x S \psi$ and absorb the $|t-r|$ weight in a Vlasov energy norm.
\item To prove boundedness on the Vlasov field using modification of the commutation vector fields of $\VV$ and certain hierarchies in the commuted equations in the spirit of \cite{FJS3} for the Einstein-Vlasov system and \cite{dim3} for the Vlasov-Maxwell system.
\end{itemize} 
Unfortunately, at the top order, we cannot fully use the null structure of the system since we do not control $ \sum_{j < k} \Omega_{jk} Z^{\gamma} \phi$, where $|\gamma|=N$ and $N$ is the maximal order of commutation. This additional difficulty then prevents us to close the energy estimates, even with the refinement of our method mentioned above. 
\end{Rq}

\subsection*{Basic notations}

In this paper we work on the $3+1$ dimensional Minkowski spacetime $(\R^{3+1},\eta)$. We will use two sets of coordinates, the Cartesian $(x^0=t,x^1,x^2,x^3)$, in which $\eta=diag(-1,1,1,1)$, and polar coordinates $(t,r,\omega_1,\omega_2)$, which are defined globally on $\R^{3+1}$ apart from the usual degeneration of spherical coordinates and at $r=0$. The hypersurface of constant $t$, for $t \geq 0$, will be denoted by
$$ \Sigma_t \hspace{2mm} := \hspace{2mm} \{ (s,x) \in \R_+ \times \R^3 \hspace{1mm} / \hspace{1mm} s=t \}.$$
In order to measure the decay rate of the velocity average of the Vlasov field, it will be convenient to use the weights
$$\tau_+:= \sqrt{1+(t+r)^2} \hspace{8mm} \text{and} \hspace{8mm} \tau_-:= \sqrt{1+(t-r)^2}.$$
We denote by $(e_1,e_2)$ an orthonormal basis on the spheres and by $\slashed{\nabla}$ the intrinsic covariant differentiation on the spheres $(t,r)=constant$. Capital Latin indices (such as $A$ or $B$) will always correspond to spherical variables. 

The velocity vector $(v^{\mu})_{0 \leq \mu \leq 3}$ is parametrized by $(v^i)_{1 \leq i \leq 3}$ and $v^0=|v|$ since we study massless particles. We denote the spherical coordinates of the velocity vector by $(v^0,v^r,v^{\omega_1},v^{\omega_2})$ and we denote by $\slashed{v}$ its spherical projection. We then have
$$\T =  v^{\mu} \partial_{\mu} , \hspace{1cm} v^r = x^i v_i, \hspace{1cm} v^{\omega_i} = <v,e_i> \hspace{1cm} \text{and} \hspace{1cm} |\slashed{v}|^2= v^{\omega_1} v_{\omega_1}+v^{\omega_2} v_{\omega_2}=v^A v_A.$$
We also introduce a spherical frame in $v$ $(\partial_{|v|}, \partial_{\theta^1_v}, \partial_{\theta^2_v})$. Let us consider an ordering on each of the sets $\Vv$ and $\VV$, so that
$$\Vv \hspace{2mm} = \hspace{2mm} \{ Z^i, \hspace{2mm} 1 \leq i \leq |\Vv| \} \hspace{1.2cm} \text{and} \hspace{1.2cm} \VV \hspace{2mm} = \hspace{2mm} \{ \widehat{Z}^i, \hspace{2mm} 1 \leq i \leq |\Vv| \}.$$
For multi-indices $\gamma \in \llbracket 1 , |\Vv | \rrbracket^p$ and $\beta \in \llbracket 1, |\VV| \rrbracket^q$, with $(p,q) \in \mathbb{N}^2$, we define $Z^{\gamma}$ and $\widehat{Z}^{\beta}$ such as
$$ Z^{\gamma} \hspace{2mm} := \hspace{2mm} Z^{\gamma_1}... Z^{\gamma_p} \hspace{1.2cm} \text{and} \hspace{1.2cm} \widehat{Z}^{\beta} \hspace{2mm} := \hspace{2mm} \widehat{Z}^{\beta_1}... \widehat{Z}^{\beta_q}.$$
If $Z \in \Vv \setminus \{S \}$, we will denote by $\widehat{Z} \in \VV$ the complete lift of $Z$, i.e.
\begin{equation}\label{complelift}
 \widehat{\partial_{\mu}} \hspace{2mm} = \hspace{2mm} \partial_{\mu} \hspace{1.2cm} \text{and} \hspace{12mm}  \widehat{\Omega_{ij}} = x^i \partial_j-x^j \partial_i+v^i \partial_{v^j}-v^j \partial_{v^i} = \widehat{\Omega}_{ij}.
\end{equation}
The subset
$$\Vv_{\Sp} \hspace{2mm} := \hspace{2mm} \{ \partial_t, \Omega_{12}, \Omega_{13}, \Omega_{23} , S \} \hspace{1cm} \text{(respectively} \hspace{4mm} \VV_{\Sp} \hspace{2mm} := \hspace{2mm} \{ \partial_t , \widehat{\Omega}_{12}, \widehat{\Omega}_{13}, \widehat{\Omega}_{23} , S, S_v \} \hspace{1mm} \text{)}$$
of $\Vv$ (respectively $\VV$) contains the commutation vector fields which preserve the spherical symmetry of a function of $x$ (respectively $(x,v)$). We refer to Subsection \ref{subsecsphrisymm} for more details.
\begin{Rq}
We point out that even if we denote by $\widehat{Z}^{\beta}$ a combination of $|\beta|$ vector fields of $\VV$, $S$ and $S_v$ are not the complete lifts of a vector field $X^{\mu} \partial_{\mu}$. We make this choice of notation for simplicity. For more details on the commutation properties between the complete lifts of Killing vector fields and Vlasov equations, we refer to \cite{FJS}, Section $2.7$ and Appendix $C$. For the purpose of this article, the information given by \eqref{complelift} will be sufficient.
\end{Rq}
As the Vlasov field will be defined on $[0,T[ \times \R^3 \times  \left( \R^3_v \setminus \{ 0 \} \right)$, we will use the notations $\R^3_x$ and $\R^3_v$ to denote respectively $\R^3$ and $\R^3 \setminus \{0\}$. We will also need to use $\chi : \R \rightarrow [0,1]$, a cutoff function satisfying
\begin{equation}\label{defchi}
 \chi = 0 \hspace{5mm} \text{on $]- \infty, \frac{1}{2}]$} \hspace{1cm} \text{and} \hspace{1cm} \chi =1  \hspace{5mm} \text{on $[1,+ \infty [$}.
\end{equation}
The notation $D_1 \lesssim D_2$ will be used for an inequality such as $ D_1 \leq C D_2$, where $C>0$ is a constant depending only on $N \in \mathbb{N}$, the maximal order of commutation, and $\delta >0$, a small number. We will raise and lower indices using the Minkowski metric $\eta$. For instance, $x_{\mu} = x^{\nu} \eta_{\nu \mu}$ so that $x_0=-t$ and $x_1=x^1$. Finally, for all sufficiently regular function $\phi : [0,T[ \times \R^3_x \rightarrow \R$, we introduce the operator $\T_{\phi}$ defined for all sufficiently regular function $f : [0,T[ \times \R^3_x \times \R^3_v \rightarrow \R$ by
$$ \T_{\phi} : f \mapsto v^{\mu} \partial_{\mu} f+\sigma v^0 \nabla_x \phi \cdot \nabla_v f.$$

\subsection*{Structure of the paper}

In section \ref{sec2phi}, we present our vector field method for massless relativistic transport equations  which does not rely on the Lorentz invariance of the operator $\T$. Section \ref{sec3phi} contains the required energy estimates and commutation formula for our study of the massless relativistic Vlasov-Poisson system. We also present our strategy in order to deal with the problems caused by the small velocities and a non zero total charge. In Section \ref{sec4phi}, we set up the bootstrap assumptions and present the strategy of the proof. In Section \ref{sec5phi}, we prove pointwise decay estimates as well as $L^2$ estimates on the potential $\phi$ and then that the particle density vanishes for small velocities. Finally, the bootstrap assumptions, which only concern the Vlasov field, are improved in Sections \ref{Secbootf} and \ref{secL2}.

\section{Pointwise decay estimates}\label{sec2phi}

\subsection{Improved decay properties given by the weights preserved by $\T$}

As mentioned earlier, the use of the weights $x^i-\frac{v^i}{v^0}t$ could be avoided (see also Remark \ref{Lorentzweight}) and we then introduce the set $\mathbf{k} := \mathbf{k}_0 \setminus \{ z_{0k}, \hspace{1mm} 1 \leq k \leq 3 \}$. For convenience, let us also introduce the weight $z$ defined by
\begin{equation}\label{defz}
 z^2 \hspace{2mm} := \hspace{2mm} \sum_{w \in \V_0} w^2 .
\end{equation}
We start by a result which illustrates the good interactions between the weights of $\V_0$ and the operator $\T$ or the commutation vector fields of $\VV$.
\begin{Lem}\label{weightpreserv}
The following properties hold.
\begin{enumerate}
\item For all $w \in \V_0$, $w$ is preserved by $\T$. More precisely, $\T(w)=0$.
\item $\V$ is preserved by the action of $\VV$. More precisely, for all $\widehat{Z} \in \VV$ and $w \in \V$, $\widehat{Z}(v^0w) \in \pm v^0\V \cup \{0\}$. Similarly, for all $\widehat{Z} \in \VV$ and $w \in \V_0$, $\widehat{Z}(v^0w) \in \pm v^0 \V_0 \cup \{0\}$.
\end{enumerate}
\end{Lem}
\begin{proof}
This follows from straightforward computations. For instance,
$$ \T ( s) = v^{\mu} \partial_{\mu} \left( t-x^i \frac{v_i}{v^0} \right)=v^0-v^i\frac{v_i}{v^0} =0 \hspace{1cm} \text{and} \hspace{1cm} S_v(s)=v^i \partial_{v^i} \left( t-x^j \frac{v_j}{v^0} \right) = -x^i \frac{v_i}{v^0}+x^j\frac{v_j v^i v_i}{|v^0|^3} = 0.$$
\end{proof}

The following properties will be useful in order to gain decay using the weights of $\V_0$. 

\begin{Lem}\label{weightdecay}
We have
$$ \forall \hspace{0.5mm} t \geq |x|, \hspace{3mm} \tau_- \leq 1+|s|, \hspace{1cm}  \forall \hspace{0.5mm} t \leq |x|, \hspace{3mm} \tau_- \leq 1 +\sum_{i=1}^3 |z_{0i}| \hspace{1cm} \text{and} \hspace{1cm} \tau_+\frac{|\slashed{v}|}{v^0} \lesssim 1+|s|+\sum_{1 \leq i < j \leq 3} |z_{ij}|.$$
\end{Lem}

\begin{proof}
The first two inequalities ensue from
$$ s =t-x^i\frac{v_i}{v^0} \geq t-|x| \hspace{1cm} \text{and} \hspace{1cm} \sum_{i=1}^3 |z_{0i}| \geq \left| x-t\frac{v}{v^0} \right| \geq |x|-t.$$
For the last one, we have $|\slashed{v}|^2=v^Av_A$ and $rv_A=<v,re_A>=C_A^{i,j}<v,\Omega_{ij}>=C_A^{i,j} z_{ij}$, where $C_A^{i,j}$ are bounded functions depending only on the spherical variables $(\omega_1,\omega_2)$. Consequently,
$$ |\slashed{v}| \leq v^0, \hspace{10mm} |\slashed{v}| \leq \frac{1}{r} \sum_{i<j} |z_{ij}| \hspace{10mm} \text{and} \hspace{10mm} |\slashed{v}| \leq  \sum_{i<j} \frac{s}{t}\frac{|z_{ij}|}{r}+\frac{x^kv_k}{trv^0}|z_{ij}| \lesssim \frac{|s|}{t}+\sum_{i<j} \frac{|z_{ij}|}{t}.$$
\end{proof}
\begin{Rq}
Note also that the Morawetz weight $\mathfrak{m}:=(t^2+r^2)\frac{v_0}{v^0}+2tx^i\frac{v_i}{v^0}$ satisfies also $\T( \mathfrak{m} )=0$. We will not use it in this article but we point out that it can be used to derive strong improved decay estimates in the null direction as
$$ v^0 - \frac{x^i}{r}v_i \hspace{2mm} \lesssim \hspace{2mm} \frac{v^0}{(t+r)^2} | \mathfrak{m} |.$$
In \cite{dim4} or \cite{massless}, we only used the inequality $v^0 - \frac{x^i}{r}v_i \hspace{2mm} \lesssim \hspace{2mm} \frac{v^0}{t+r} z$.
\end{Rq}
These two lemmas directly imply that
\begin{Lem}\label{zprop}
$\T(z)=0$. Moreover,
$$\forall \hspace{0.5mm} (\ZZ , p) \in \VV \times \R_+, \hspace{2cm} |\ZZ (z^p) | \lesssim z^p \hspace{1cm} \text{and} \hspace{1cm} \forall \hspace{0.5mm} (t,x) \in \R_+ \times \R^3, \hspace{1cm} \tau_- \leq z.$$
\end{Lem}

In Section $5$ of \cite{dim4}, we proved that $Y_{ij}:=v^i \partial_j-v^j \partial_i$ is a null form. The purpose of the following lemma is to recover this result without using the Lorentz boosts.
\begin{Pro}\label{null}
Let $u : [0,T[ \times \R^3_x \rightarrow \R $ be a sufficiently regular function and $1 \leq i < j \leq 3$. We have
$$ | Y_{ij} u | \hspace{2mm} = \hspace{2mm} \left| \frac{v^i}{v^0} \partial_ju-\frac{v^j}{v^0} \partial_iu \right| \hspace{2mm} \lesssim \hspace{2mm} \frac{1}{\tau_+} \left(\sum_{ k < l} | \Omega_{kl} u |+\sum_{w \in \V} |w \nabla_x u| \right).$$
\end{Pro}
\begin{proof}
Without loss of generality, we suppose that $(i,j)=(1,2)$. Let us fix $(t,x) \in [0,T[ \times \R^3$ and note that the inequality is straightforward if $t+|x| \leq 1$. Now, remark that
\begin{eqnarray}
\nonumber v^1 \partial_2-v^2 \partial_1 & = & \frac{1}{x^3} \left(-v^1 \Omega_{23}+v^1x^2 \partial_3+v^2 \Omega_{13}-v^2x^1 \partial_3 \right) = \frac{1}{x^3} \left(-v^1 \Omega_{23}+v^2 \Omega_{13}-v^0z_{12} \partial_3 \right), \\ \nonumber
 v^1 \partial_2-v^2 \partial_1 & = & \frac{1}{x^1} \left(v^1 \Omega_{12}+v^1x^2 \partial_1-v^2 x^1 \partial_1 \right) = \frac{1}{x^1} \left(v^1 \Omega_{12}-v^0z_{12} \partial_1 \right), \\ \nonumber
 v^1 \partial_2-v^2 \partial_1 & = & \frac{1}{x^2} \left(v^2 \Omega_{12}-v^2x^1 \partial_2+v^1 x^2 \partial_2 \right) = -\frac{1}{x^2} \left(v^2 \Omega_{12}-v^0z_{12} \partial_2 \right).
\end{eqnarray}
As there exists $1 \leq k \leq 3$ such that $x^k \geq \frac{1}{\sqrt{3}}|x|$, we obtain
\begin{equation}\label{eq:inr}
 \left| \frac{v^1}{v^0} \partial_2 u-\frac{v^2}{v^0} \partial_1 u \right| \hspace{2mm} \lesssim \hspace{2mm} \frac{1}{r} \left(\sum_{k < l} | \Omega_{kl} u |+ |z_{kl} \nabla_x u| \right).
 \end{equation}
Finally, the decay in the variable $t$ can be obtained combining \eqref{eq:inr} with
$$ \left|\frac{z_{kl}}{r} \right|=\left| \frac{sz_{kl}}{tr}+\frac{rv^rz_{kl}}{rtv^0} \right| \leq 2\frac{|s|}{t}+\frac{|z_{kl}|}{t} \hspace{5mm} \text{and} \hspace{5mm} \left|\frac{\Omega_{kl}}{r}u \right|=\left| \frac{s\Omega_{kl}}{tr}u+\frac{rv^r\Omega_{kl}}{rtv^0}u \right| \leq 2\frac{|s||\nabla u|}{t}+\frac{|\Omega_{kl} u|}{t}.$$
\end{proof}

\subsection{Decay estimates for velocity averages}
In this section we prove a functional inequality adapted for solutions to massless relativistic transport equations. We consider two sufficiently regular functions $g : [0,T[ \times \R^3_x \times \R^3_v \rightarrow \R$ and $\phi : [0,T[ \times \R^3_x \rightarrow \R$ and we start by proving a commutation property between our vector fields and the averaging in $v$.
\begin{Lem}\label{Comaverage}
$$ \forall \hspace{0.5mm} Z \in \Vv \setminus \{S \}, \hspace{1cm} \left| Z \int_{\R^3_v} |g| dv \right| \hspace{1mm} \leq \hspace{1mm} \int_{\R^3_v} |\ZZ g | dv \hspace{5mm} \text{and} \hspace{5mm} \left| S \int_{\R^3_v} |g| dv \right| \hspace{1mm} \leq \hspace{1mm} \int_{\R^3_v} |S g | dv .$$
\end{Lem}
\begin{proof}
Consider for instance $\Omega_{ij}=\widehat{\Omega}_{ij}-v^i \partial_{v^j}+v^j \partial_{v^i}$. We have, integrating by parts in $v$,
$$  \Omega_{ij}\left( \int_{\R^3_v } |g| dv \right) = \int_{\R^3_v}  \widehat{\Omega}_{ij} (|g|) dv -\int_{\R^3_v}  \left( v^i\partial_{v^j} -v^j \partial_{v^i}  \right)(|g|) dv= \int_{\R^3_v}  \widehat{\Omega}_{ij} (|g|) dv$$
and it remains to note that $\widehat{\Omega}_{ij} (|g|)=\frac{g}{|g|} \widehat{\Omega}_{ij} g$ and to use the triangle inequality.
\end{proof}
The next Lemma contains two analogous formulas. The first one is simpler whereas the second one seems to be easier to adapt in the context of a curved background such as the Schwarzschild spacetime since no spatial translation are involved.
\begin{Lem}\label{angupart}
We have
$$ \left| \frac{v^i}{v^0} \partial_i (g) -\frac{v^r}{v^0} \partial_r (g) \right| \hspace{0.5mm} \lesssim \hspace{0.5mm} \frac{1}{r} \sum_{i < j} |z_{ij}| \left| \nabla_x g \right|, \hspace{1.8cm}  \left| \int_{\R^3_v}  \frac{v^i}{v^0} \partial_i (|g|) -\frac{v^r}{v^0} \partial_r (|g|) dv \right| \hspace{0.5mm} \lesssim \hspace{0.5mm} \frac{1}{r}  \sum_{i<j} \int_{ \R^3_v}|g|+ | \widehat{\Omega}_{ij} g | dv.$$
\end{Lem}
\begin{proof}
For the first formula, Lemma \ref{weightdecay} gives us
$$\left| v^i \partial_i (g) -v^r \partial_r (g) \right| \hspace{2mm} = \hspace{2mm} \left| v^A e_A (g) \right|  \hspace{2mm} \lesssim \hspace{2mm} \frac{v^0}{r} \sum_{i < j} |z_{ij}| \cdot |\nabla_x g |.$$
For the second one, recall that $rv_A= C^{i,j}_A v^0z_{ij}$ and $r e_A = C_A^{i,j} \Omega_{ij}$, with $C^{i,j}_A$ bounded functions depending only on the spherical variables $(\omega_1,\omega_2)$. Hence,
\begin{eqnarray}
\nonumber \left| \int_{\R^3_v}  \frac{v^i}{v^0} \partial_i (|g|) -\frac{v^r}{v^0} \partial_r (|g|) dv \right| & = & \left| \int_{\R^3_v}  \frac{v^A}{v^0} e_A (|g|)  dv \right| \hspace{2mm} \lesssim \hspace{2mm} \frac{1}{r} \sum_{i < j} \sum_{k<l} \left| \int_{\R^3_v}  \frac{z_{ij}}{r} \Omega_{kl} (|g|)  dv \right| \\ \nonumber 
& \lesssim & \frac{1}{r} \sum_{i=1}^3 \sum_{k<l} \left| \int_{\R^3_v} \frac{v^i}{v^0}  \widehat{\Omega}_{kl} (|g|)  dv \right|+ \frac{1}{r} \sum_{i =1}^3 \sum_{k<l} \left| \int_{\R^3_v}  \frac{v^i}{v^0} (v^k \partial_{v^l}-v^l \partial_{v^k}) (|g|)  dv \right|.
\end{eqnarray}
It remains to integrate by parts in the last integral, use the triangle inequality and notice that $\left| \partial_{v^l} \left( \frac{v^iv^k}{v^0} \right) \right| \lesssim 1$.
\end{proof}
Recall that the proof of the Klainerman-Sobolev \eqref{th6} crucially relies on the equality
$$(t-r)\partial_r = -\frac{r}{t+r}S+\frac{tx^i}{(t+r)r}\Omega_{0i}.$$
The purpose of the following Lemma is to obtain a similar identity which does not involve any Lorentz boosts $\Omega_{0i}$.
\begin{Lem}\label{recovertauminus}
We have,
 $$ \tau_- \left| \partial_r \int_{\R^3_v} |g| dv \right| \hspace{2mm} \lesssim \hspace{2mm} r  \int_{\R^3_v} \left| \T_{\phi}(g)  \right| \frac{dv}{v^0}+\int_{\R^3_v} |s \nabla_{t,x} g | dv+ \int_{\R^3_v}|g|dv+\sum_{\ZZ \in \VV}\int_{\R^3_v} | \ZZ g| dv .$$
\end{Lem}
\begin{proof}
Start by noticing that
\begin{eqnarray}
\nonumber  (t^2-r^2)  \partial_t g \hspace{2mm} = \hspace{2mm} ts \partial_t g+tr \frac{v^r}{v^0} \partial_t g-r^2 \partial_t g & = & ts \partial_t g +r \frac{v^r}{v^0} Sg-r^2 \left(\frac{v^r}{v^0} \partial_r g+\partial_t g \right) \\ \nonumber
& = & ts \partial_t g+r \frac{v^r}{v^0} Sg-\frac{r^2}{v^0} \T(g)+r^2 \left( \frac{v^i}{v^0} \partial_i g-\frac{v^r}{v^0} \partial_r g \right) .
\end{eqnarray}
We then deduce, using Lemmas \ref{Comaverage} and \ref{angupart}, that
\begin{eqnarray}
 \nonumber \tau_- \left| \partial_t \int_{\R^3_v} |g| dv \right| & \lesssim & \left| \partial_t \int_{\R^3_v} |g| dv \right|+ \frac{1}{\tau_+}\left| \int_{\R^3_v} \frac{g}{|g|} \left( ts \partial_t g+r \frac{v^r}{v^0} Sg-\frac{r^2}{v^0} \T(g)+r^2 \left( \frac{v^i}{v^0} \partial_i g-\frac{v^r}{v^0} \partial_r g \right) \right) dv \right|  \\
 & \lesssim & r \left| \int_{\R^3_v} \T(|g|) \frac{dv}{v^0} \right|+\int_{\R^3_v} |s \partial_t g | dv+ \int_{\R^3_v} |g|dv+ \sum_{\ZZ \in \VV}\int_{\R^3_v} | \ZZ g| dv . \label{taupartialt}
 \end{eqnarray}
To conclude the proof, remark that
\begin{eqnarray}
\nonumber \left| \int_{\R^3_v} \T(|g|) \frac{dv}{v^0} \right| & \leq & \left| \int_{\R^3_v} \T_{\phi}(|g|) \frac{dv}{v^0} \right|+\left| \int_{\R^3_v} -\sigma \nabla_x \phi \cdot \nabla_v |g| dv \right| \hspace{2mm} \leq \hspace{2mm} \int_{\R^3_v} \left| \T_{\phi}(g)\right| \frac{dv}{v^0}, \\
\nonumber (t-r)  \partial_r g & = & s \partial_r g+r\frac{v^r}{v^0} \partial_r g-r \partial_r g \hspace{2mm} = \hspace{2mm} s \partial_r g+\left(\frac{v^r}{v^0}-1 \right) Sg+\left(1-\frac{v^r}{v^0} \right) t \partial_t g \\ \nonumber
& = & s \partial_r g+\left(\frac{v^r}{v^0}-1 \right) Sg+s\partial_t g+(r-t)\frac{v^r}{v^0} \partial_t g,
\end{eqnarray}
and use Lemma \ref{Comaverage} as well as $|v^r| \leq v^0$ and \eqref{taupartialt}.
\end{proof}
\begin{Rq}
Note that we only used the space translations in order to control $\partial_r g$ in the region $r \leq 1$. If $r \geq 1$, we could take advantage of the relation
$$ \partial_r  \hspace{2mm} = \hspace{2mm} \frac{S}{r}-\frac{t}{r} \partial_t \hspace{2mm} = \hspace{2mm} \frac{S}{r}-\frac{s}{r} \partial_t+\frac{v^r}{v^0} \partial_t.$$
\end{Rq}
Before proving pointwise decay estimates for $\int_{\R^3_v} |g| dv$, let us recall two classical Sobolev inequalities for velocity averages (on $\R^3_x$ and on the sphere $\mathbb{S}^2$).

\begin{Lem}\label{Sobsphere}
Let $H : \R^3_x \times  \R^3_v \rightarrow \R$ and $h : \mathbb{S}^2 \times  \R^3_v \rightarrow \R$ be sufficiently regular functions. Then, with $\widehat{\Omega}^{\beta} \in \{ \widehat{\Omega}_{12}, \widehat{\Omega}_{13}, \widehat{\Omega}_{23} \}^{|\beta|}$,
\begin{eqnarray}
\nonumber \left\| \int_{\R^3_v}  |H| dv \right\|_{L^{\infty}(\R^3_x)} & \leq  & \left\| \int_{\R^3_v}   \left| \partial_1 \partial_2 \partial_3 H \right| dv \right\|_{L^1(\R^3_x)},  \\ \nonumber
\left\| \int_{\R^3_v}  |h| dv \right\|_{L^{\infty}(\mathbb{S}^2)} & \lesssim  & \sum_{ |\beta| \leq 2 } \left\| \int_{\R^3_v}   \left| \widehat{\Omega}^{\beta} h \right| dv \right\|_{L^1(\mathbb{S}^2)} .
\end{eqnarray}
\end{Lem}
\begin{proof}
We refer to Lemma $3.5$ of \cite{massless} for a proof of the second inequality. For the first one, use successively three times a one dimensional Sobolev inequality and then Lemma \ref{Comaverage} in order to get
\begin{eqnarray}
\nonumber \int_{\R^3_v}  |H|(x,v) dv & = & \int^{x^1}_{- \infty} \partial_1 \left( \int_{\R^3_v}  |H|(y^1,x^2,x^3,v) dv \right) dy^1 \hspace{2mm} \leq \hspace{2mm} \int^{x^1}_{- \infty} \int_{\R^3_v} |\partial_1 H|(y^1,x^2,x^3,v) dv dy^1 \\ \nonumber
& = & \int^{x^1}_{- \infty} \int_{-\infty}^{x^2} \partial_2 \left( \int_{\R^3_v}  |\partial_1 H|(y^1,y^2,x^3,v) dv \right) dy^2 dy^1 \\ \nonumber & \leq & \int^{x^1}_{- \infty} \int_{- \infty}^{x^2} \int_{\R^3_v} | \partial_2 \partial_1 H|(y^1,y^2,x^3,v) dv dy^2 dy^1 \\ \nonumber
& = & \int^{x^1}_{- \infty} \int_{-\infty}^{x^2} \int_{-\infty}^{x^3} \partial_3 \left( \int_{\R^3_v}  |\partial_2 \partial_1 H|(y^1,y^2,y^3,v) dv \right) dy^3 dy^2 dy^1 \hspace{1mm} \leq \hspace{1mm} \int_{\R^3_y} \int_{\R^3_v} | \partial_3 \partial_2 \partial_1 H| dv dy.
\end{eqnarray}
\end{proof}
We are now ready for the main result of this section.

\begin{Pro}\label{decay}
For all $(t,x) \in [0,T[ \times \R^3$ and $p \geq 0$, we have, with $\widehat{\Omega}^{\beta} \in \{ \widehat{\Omega}_{12}, \widehat{\Omega}_{13}, \widehat{\Omega}_{23} \}^{|\beta|}$ and $\ZZ^{\xi} \in \VV^{|\xi|}$,
\begin{eqnarray}
\nonumber \int_{\R^3_v} |g|(t,x,v) dv & \lesssim & \frac{1}{(1+|x|)^2 \tau_-} \left( \sum_{|\beta| \leq 2}  \int_{\Sigma_t} \int_{\R^3_v} r \left| \widehat{\Omega}^{\beta} \left( \T_{\phi} ( g) \right) \right| \frac{dv}{v^0} dx+\sum_{|\xi| \leq 3}  \int_{\Sigma_t} \int_{\R^3_v} (1+|s|)\left| \ZZ^{\xi} g \right| dv dx \right), \\ \nonumber
\int_{\R^3_v} |g|(t,x,v) dv & \lesssim & \frac{1}{(1+|x|)^2 \tau_-^{1+p}} \left( \sum_{ |\beta| \leq 2 }   \int_{\Sigma_t} \int_{\R^3_v} r \left| z^p \widehat{\Omega}^{\beta} \left( \T_{\phi} ( g) \right) \right| \frac{dv}{v^0} dx +  \sum_{ |\xi| \leq 3 } \int_{\Sigma_t} \int_{\R^3_v} \left| z^{p+1} \ZZ^{\xi} g \right| dv dx \right).
\end{eqnarray}
\end{Pro}
\begin{proof}
We fix $p \geq 0$, $(t,x) \in [0,T[ \times \R^3$ and we start by considering the case where $|x| \leq 1$. Then, as $\tau_- \leq 1+|s|$ by Lemma \ref{weightdecay} and using the $L^1(\R^3)-L^{\infty}(\R^3)$ Sobolev inequality on $\R^3_x$ of Lemma \ref{Sobsphere}, we have
$$ \tau_-^{1+p} \int_{\R^3_v} |g| dv \hspace{2mm} \lesssim \hspace{2mm} \int_{\R^3_v} |g|+|s^{1+p} g| dv \hspace{2mm} \lesssim \hspace{2mm} \sum_{|\beta| \leq 3} \int_{\Sigma_t} \int_{\R^3_v} (1+|s|)^{1+p}|\partial^{\beta}_x g| dv dx.$$
Otherwise, $|x| \geq 1$ and Lemma \ref{recovertauminus} allows us to obtain
\begin{eqnarray}
\nonumber \tau_-^{1+p} |x|^2 \int_{\R^3_v} |g|(t,|x| \omega,v) dv & = & -|x|^2 \int_{r=|x|}^{+ \infty} \partial_r \left( \tau_-^{1+p} \int_{\R^3_v} |g|(t,r \omega,v) dv \right) dr \\ \nonumber
& \lesssim & \int_{r=|x|}^{+ \infty} \left( \int_{\R^3_v} \tau_-^p |g| dv+ \tau_-^{1+p} \left| \partial_r \int_{\R^3_v} |g|(t,r \omega,v) dv \right| \right) r^2dr \\ \nonumber
& \lesssim & \int_{r=0}^{+ \infty} \tau_-^p \left( r  \int_{\R^3_v} \left| \T_{\phi}(g)  \right| \frac{dv}{v^0}+ \int_{\R^3_v} |s \nabla_{t,x} g | dv+ \sum_{|\kappa| \leq 1} \int_{\R^3_v} | \ZZ^{\kappa} g| dv \right) r^2 dr.
\end{eqnarray}
It remains to apply the $L^1(\mathbb{S}^2)-L^{\infty}(\mathbb{S}^2)$ Sobolev inequality of Lemma \ref{Sobsphere} to $\frac{1}{v^0}\T_{\phi}(g)$, $g$, $s \nabla_{t,x} g$ as well as $\ZZ g$ and then to use the inequality $\tau_-^p \leq z^p$ which comes from Lemma \ref{zprop}.
\end{proof}
\begin{Rq}
We could precise these estimates as follows. In the first one, the proof gives that the weight $s$ only hits derivatives $\ZZ^{\xi} g$ composed of at least one translation and a similar observation can be done for the second estimate. These properties can be useful in order to exploit certain hierarchies in the commuted equations when one studies a Vlasov system (see \cite{dim3} and \cite{ext}) but will not be used in this paper.
\end{Rq}
\begin{Rq}
If one wants to fully recover the decay rate given by a standard Klainerman-Sobolev inequality as \eqref{th6}, which relies on Lorentz boosts, for a function $f$ solution so $\T(f)=0$, one needs to take $p =2$ and then assume more decay initially on $f$ (two powers of $|x|$).
\end{Rq}
\begin{Rq}\label{Lorentzweight}
The boosts $x^i-t\frac{v^i}{v^0}$ are used here to gain decay in $t-r$ in the region $r \geq t$. We could prove a similar result without using them by propagating weighted $L^1$ norms in the region $r \geq t$ with the multiplier $\tau_-^p \partial_t$. Note that it would not require stronger initial decay on $f$ and we then choose to work with $z_{0i}$ for simplicity.
\end{Rq}

\section{Preliminaries for the study of the relativistic massless Vlasov-Poisson system}\label{sec3phi}

We consider, for all this section, $(f,\phi)$ a sufficiently regular solution to \eqref{VP1}-\eqref{VP2} defined on $[0,T[$.

\subsection{Commutation properties and energy estimates}

In order to use a vector field method to study the Vlasov field, we will have to commute both equations.
\begin{Pro}\label{comu0}
Let $ Z \in \Vv \setminus \{ S \}$. We have
$$ \Delta Z \phi \hspace{2mm} = \hspace{2mm} \int_{\R^3_v} \widehat{Z} f dv, \hspace{10mm} \Delta S \phi \hspace{2mm} = \hspace{2mm} \int_{\R^3_v} S f dv+2 \int_{\R^3_v} f dv \hspace{10mm} \text{and} \hspace{10mm} \T_{\phi}(\widehat{Z} f) \hspace{2mm} = \hspace{2mm} - \sigma v^0 \nabla_x Z \phi \cdot \nabla_v f .$$
For the scaling in $(t,x)$ and the one in $v$, we have
$$ \T_{\phi} ( S f) \hspace{2mm} = \hspace{2mm} -\sigma v^0 \nabla_x S \phi \cdot \nabla_v f+2 \sigma v^0 \nabla_x  \phi \cdot \nabla_v f  \hspace{12mm} \text{and} \hspace{12mm}  \T_{\phi}( S_v f ) \hspace{2mm} = \hspace{2mm} 3 \sigma v^0 \nabla_x Z \phi \cdot \nabla_v f.$$
\end{Pro}
\begin{proof}
To commute the Poisson equation, we use that $[\Delta,Z]=0$, $[\Delta,S]=2 \Delta$ as well as
$$ \partial_{\mu} \int_{\R^3_v} f dv = \int_{\R^3_v} \partial_{\mu} f dv, \hspace{1cm} S \int_{\R^3_v}  f dv = \int_{\R^3_v} S f dv \hspace{10mm} \text{and} \hspace{10mm} \Omega_{ij} \int_{\R^3_v}  f dv = \int_{\R^3_v} \widehat{\Omega}_{ij} f dv, $$
which ensues from integration by parts in $v$. For the Vlasov equation, one can check by direct computations that $[\T,\widehat{Z}]=0$, $[\T,S]=\T$ and $[\T, S_v]=-\T$. Then, note for instance that
\begin{eqnarray}
\nonumber S_v \left( v^0 \nabla_x \phi \cdot \nabla_v f \right) & = & v^i \partial_{v^i}(v^0) \nabla_x \phi \cdot \nabla_v f+  v^0\nabla_x   \phi \cdot \nabla_v S_v f-v^0\partial_{v^i}(v^i) \nabla_x   \phi \cdot \nabla_v f \\ \nonumber
& = & v^0 \nabla_x   \phi \cdot \nabla_v S_v f-2 v^0 \nabla_x   \phi \cdot \nabla_v f .
\end{eqnarray}
\end{proof}
Iterating this Proposition, one obtains
\begin{Pro}\label{comu}
Let $Z^{\gamma} \in \Vv^{|\gamma|}$ and $\widehat{Z}^{\beta} \in \VV^{|\beta|}$. Then, there exist integers $C^{\gamma}_{\xi}$ and $C^{\beta}_{\kappa, \alpha}$ such that 
\begin{eqnarray}
\nonumber [T_{\phi}, \widehat{Z}^{\beta}](f) \hspace{2mm} = \hspace{2mm} T_{\phi} ( \widehat{Z}^{\beta} f ) & = & \sum_{|\kappa|+|\alpha| \leq |\beta|} C^{\beta}_{\kappa,\alpha}v^0 \nabla_x Z^{\kappa} \phi \cdot \nabla_v \widehat{Z}^{\alpha} f, \\
\nonumber \Delta Z^{\gamma} \phi & = & \sum_{|\xi| \leq |\gamma|} C^{\gamma}_{\xi} \int_{\R^3_v} \widehat{Z}^{\xi} f dv.
\end{eqnarray}
\end{Pro}

In order to control the $L^1_{x,v}$ norm of the derivatives of $f$, we will apply several times the following approximate conservation law.

\begin{Pro}\label{energyparti}
Let $g:[0,T[ \times \R^3_x \times \R^3_v \rightarrow \R$ be a sufficiently regular function. Then,
$$\forall \hspace{0.5mm} t \in [0,T[, \hspace{1cm} \left\| g (t, \cdot , \cdot ) \right\|_{L^1_{x,v}} \hspace{2mm} \leq \hspace{2mm} \left\| g (0, \cdot , \cdot ) \right\|_{L^1_{x,v}} + \int_0^t \int_{\Sigma_s} \int_{\R^3_v} \left| \T_{\phi}(g) \right| \frac{dv}{v^0} dx ds.$$ 
\end{Pro}
\begin{proof}
Using the definition of $\T_{\phi}$ and integration by parts in $x$ and in $v$, one has
\begin{eqnarray}
\nonumber \int_0^t \int_{\Sigma_s} \int_{\R^3_v} \T_{\phi}(|g|)  \frac{dv}{v^0} dx ds \hspace{-1mm} & = & \hspace{-1mm} \int_0^t \left( \int_{\Sigma_s} \int_{\R^3_v} \partial_t |g| dv dx + \int_{\R^3_v} \int_{\Sigma_s} \frac{v^i}{v^0} \partial_i |g| dv dx  + \sigma \int_{\Sigma_s} \partial^i \phi \int_{\R^3_v} \partial_{v^i} |g| dv dx \right) ds \\ \nonumber
& = & \hspace{-1mm} \int_0^t \partial_t \int_{\Sigma_s} \int_{\R^3_v} |g| dv dx ds \hspace{2mm} = \hspace{2mm} \int_{\Sigma_t} \int_{\R^3_v} |g| dv dx -\int_{\Sigma_0} \int_{\R^3_v} |g| dv dx .
\end{eqnarray}
It remains to use the triangle inequality and that $ \left| \T_{\phi}(|g|) \right| = \left| \frac{g}{|g|} \T(g) \right| \leq \left| \T_{\phi}(g) \right|$.
\end{proof}

\subsection{The problem caused by the small velocities}

In order to circumvent any problem with the characteristics of the operator $\T_{\phi}$ reaching the value $v=0 \notin \R^3_v$, we introduce 
$$ \T_{\phi}^{\chi} : g \mapsto v^{\mu} \partial_{\mu}g + \chi(|v|) \sigma v^0 \nabla_x \phi \cdot \nabla_v g, $$
where $\chi$ is the cutoff function defined by \eqref{defchi}, and we will first define our solution $(f,\phi)$ as a solution to the system
\begin{eqnarray}
 \T_{\phi}^{\chi} (f) & = & 0, \label{VP1bis} \\
\Delta \phi & = & \int_{\R^3_v} f dv. \label{VP2bis}
\end{eqnarray}
We will suppose that $f(0, \cdot, v)$ vanishes for all $|v| \leq 2$ and one step of the proof will consist in proving that 
\begin{equation}\label{eq:goalsupp}
\forall \hspace{0.5mm} |v| \leq 1, \hspace{1cm} f( \cdot , \cdot, v) \hspace{1mm} = \hspace{1mm} 0, \hspace{1cm} \text{ so that} \hspace{1cm} \T_{\phi}(f) \hspace{2mm} = \hspace{2mm} \T_{\phi}^{\chi} (f) \hspace{2mm} = \hspace{2mm} 0.
\end{equation}
For this, we will study the characteristics $(X,V)$ of the operator $\T_{\phi}^{\chi}$, which satisfy
$$ \dot{X}(t) \hspace{2mm} = \hspace{2mm} \frac{V(t)}{|V|(t)}, \hspace{2cm} \dot{V}(t) \hspace{2mm} = \hspace{2mm} \chi \Big( |V|(t) \Big) \hspace{0.5mm} \nabla_x \hspace{0.3mm} \phi \left( t,X(t) \right).$$
In view of the definition of $\chi$, none of the characteristics of $\T^{\chi}_{\phi}$ can reach $v=0$ since $\dot{V} =0$ if $|V| < 1$. Consequently, we can use the method of the characteristics to express $f(t,\cdot , \cdot )$ in terms of $f(0,\cdot , \cdot )$. More precisely, for each $(t,x,v) \in [0,T[ \times \R^3_x \times \R^3_v$, there exists a characteristic $(X,V)$ of $\T_{\phi}^{\chi}$ which is well defined on $[0,t]$, which takes its value in $\R^3_x \times \R^3_v$ and such that $(X(t),V(t))=(x,v)$, so that
\begin{equation}\label{charac}
 f(t,x,v) \hspace{2mm} = \hspace{2mm} f(t,X(t),V(t)) \hspace{2mm} = \hspace{2mm} f(0,X(0),V(0)).
\end{equation}
The problem of the operator $\T_{\phi}$ is that certain of its characteristics can reach $v=0$ and because of that,
\begin{itemize}
\item if the initial data do not vanish for small velocities, the massless relativistic Vlasov-Poisson could fail to admit a local $C^1$ solution (see Section $8$ of \cite{dim4}).
\item We cannot prove directly that \eqref{eq:goalsupp} holds using the characteristics $(\overline{X},\overline{V})$ of the operator $T_{\phi}$. Indeed, if $\overline{V}$ reaches the value $v=0 \notin \R^3_v$, no formula such as \eqref{charac} holds for $(\overline{X},\overline{V})$ since the characteristic would not be well defined for $t=0$.
\end{itemize}

\subsection{The problem caused by the non zero total charge}\label{subseccharge}

The total charge of the plasma
$$ Q \hspace{2mm} := \hspace{2mm} \int_{\R^3_x} \int_{\R^3_v} f dv dx$$
is a conserved quantity in time. If $Q \neq 0$, i.e. if the plasma is not electrically neutral, the potential $\phi$ has restricted decay properties. Indeed, denoting by $\mathbb{S}_{t,r}$ the sphere of radius $r$ centered in the origin of $\Sigma_t$, the divergence theorem gives
$$Q_{\phi}(t) \hspace{2mm} := \hspace{2mm} \lim_{r \rightarrow + \infty} \int_{\mathbb{S}_{t,r}} \partial_r \phi d \mathbb{S}_{t,r} \hspace{2mm} = \hspace{2mm} \hspace{2mm} \int_{\Sigma_t} \int_{\R^3_v} \Delta \phi dv dx \hspace{2mm} = \hspace{2mm} Q,$$
which implies that $\nabla_x \phi$ cannot decay faster than $r^{-2}$ if $Q \neq 0$. Consequently, one cannot expect to prove a better $L^2$ estimate than
$$ \| \nabla_x \phi \|_{L^2(\Sigma_t)} \hspace{2mm} \lesssim \hspace{2mm} \frac{\epsilon}{\sqrt{1+t}}.$$
In contrast, if the total charge $Q_{\psi}$ of a sufficiently regular function $\psi :[0,T[ \times \R^3 \rightarrow \R$ is equal to zero, one can expect $\nabla_x \psi$ to decay faster than $r^{-2}$. This will allow us to prove that the $L^2(\Sigma_t)$ norms of the chargeless derivatives of the potential $\phi$ decay as $(1+t)^{-1}$. The following lemma is then particularly interesting.
\begin{Pro}\label{Procharge}
For all $\widehat{Z}^{\gamma} \in \VV^{|\gamma|}$, we have
\begin{itemize}
\item $Q_{Z^{\gamma} \phi} = (-1)^{|\gamma|} Q_{\phi}$ if $Z^{\gamma} = S^{|\gamma|}$,
\item $Q_{Z^{\gamma} \phi} =0$ otherwise.
\end{itemize}
\end{Pro}
\begin{proof}
Let $Z \in \Vv \setminus \{S \}$. Then, by the commutation formula of Proposition \ref{comu0} and the divergence theorem,
$$Q_{Z \phi}(t) \hspace{2mm} = \hspace{2mm}  \int_{\Sigma_t} \int_{\R^3_v} \Delta Z \phi \hspace{0.5mm} dv dx \hspace{2mm} = \hspace{2mm} \int_{\Sigma_t} \int_{\R^3_v} \widehat{Z} f \hspace{0.5mm} dv dx.$$
If $Z = \partial_k$, $1 \leq k \leq 3$ (respectively $Z= \Omega_{ij}$, $1 \leq i < j \leq 3$), integration by parts in $x$ (respectively $x$ and $v$) gives $Q_{Z \phi}(t)=0$. If $Z = \partial_t$, then, using $\T_{\phi}(f)=0$, we get
\begin{equation}\label{eq:partialt}
 \int_{\Sigma_t} \int_{\R^3_v} \partial_t f \hspace{0.5mm} dv dx \hspace{2mm} = \hspace{2mm}  -\int_{\R^3_v} \frac{v^i}{v^0} \int_{\Sigma_t} \partial_i f \hspace{0.5mm} dv dx-\sigma \int_{\Sigma_t} \nabla_x \phi \cdot \int_{\R^3_v}  \nabla_v f \hspace{0.5mm} dv dx \hspace{2mm} = \hspace{2mm} 0 . 
\end{equation}
Applying Proposition \ref{comu0}, integrating by parts in $x$ and using \eqref{eq:partialt}, we get 
$$Q_{S \phi}(t)  \hspace{2mm} = \hspace{2mm} \int_{\Sigma_t} \int_{\R^3_v} t \partial_t f+x^i \partial_i f+2f \hspace{0.5mm} dv dx \hspace{2mm} = \hspace{2mm} - \int_{\Sigma_t} \int_{\R^3_v} f \hspace{0.5mm} dv dx .$$
The general case can be treated similarly making an induction on $|\gamma|$ and using the commutation formula of Proposition \ref{comu}) (we refer to the Appendix $C$ of \cite{massless} for a more detailed proof in the context of the Vlasov-Maxwell system). 
\end{proof}

\subsection{Derivatives preserving the symmetries of the solutions}\label{subsecsphrisymm}

\begin{Def}
Let $g : [0,\overline{T}] \times \R^3_x \times \R^3_v \rightarrow \R$ and $\psi : [0,\overline{T}] \times \R^3_x \rightarrow \R$, with $\overline{T} \geq 0$, be sufficiently regular function. We say that $g$ (respectively $\psi$) is spherically symmetric if
$$ \forall \hspace{0.5mm} 1 \leq i < j \leq 3, \hspace{1cm} \widehat{\Omega}_{ij} g \hspace{2mm} = \hspace{2mm} 0 \hspace{1cm} \text{(respectively} \hspace{0.4cm} \Omega_{ij} \psi \hspace{2mm} = \hspace{2mm} 0 \hspace{1mm} \text{).}$$
In that case, $\psi$ does not depend on the spherical variables $(\omega_1, \omega_2)$ and $g$ is a function of $(t,|x|,|v|,x \cdot v)$. We say that $(f,\phi)$ is a spherically symmetric solution to \eqref{VP1}-\eqref{VP2} if $f$ and $\phi$ are both spherically symmetric.
\end{Def}
The following properties will be useful for the remaining of this paper.
\begin{Pro}\label{Prospheri}
If $f(0,\cdot, \cdot )$ is spherically symmetric, then $(f,\phi)$ is a spherically symmetric solution to \eqref{VP1}-\eqref{VP2}. Moreover, 
\begin{itemize}
\item for all $\widehat{Z}^{\beta} \in \VV_{\Sp}^{|\beta|}$, $\widehat{Z}^{\beta} f$ is spherically symmetric provided that it is well defined.
\item For $Z^{\gamma} \in \Vv_{\Sp}^{|\gamma|}$, $Z^{\gamma} \phi$ is spherically symmetric provided that it is well defined.
\end{itemize}
\end{Pro}
\begin{proof}
According to Proposition \ref{comu}, we have, for $\Omega$ a rotational vector field,
$$ \T(\widehat{\Omega}  f) \hspace{2mm} = - \sigma v^0 \nabla_x \Omega \phi \cdot \nabla_v f \hspace{1cm} \text{and} \hspace{1cm} \Delta \Omega \hspace{2mm} = \hspace{2mm} \int_{\R^3_v} \widehat{\Omega} f dv.$$
We then deduce that $\widehat{\Omega} f$ and $\Omega \phi$ vanish since it is initially true, which imply that $(f, \phi)$ is spherically symmetric. The last two points of the Proposition then ensue from $[ \Omega, S] = [ \Omega, \partial_t ] = [\widehat{\Omega}, S] = [\widehat{\Omega}, S_v]=[ \widehat{\Omega}, \partial_t ] =0$.
\end{proof}

\section{Bootstrap assumptions and strategy of the proof}\label{sec4phi}

In view of Proposition \ref{decay}, we introduce for all $(M,n) \in \mathbb{N} \times \mathbb{N}^*$ the following energy norms
$$\E_M^n[g](t) \hspace{2mm} = \hspace{2mm} \sum_{ |\alpha|+|\beta| \leq M } \int_{\Sigma_t} \int_{\R^3_v} r z^{n-1} \left| \widehat{Z}^{\beta} \left( \T_{\phi} \left( \widehat{Z}^{\alpha} g \right) \right) \right| \frac{dv}{v^0} dx + \sum_{ |\xi| \leq M } \int_{\Sigma_t} \int_{\R^3_v} \left| z^n \ZZ^{\xi} g \right| dv dx.$$
The goal of the remaining of this article is to prove Theorem \ref{theorem}. We then consider $f_0$ a function satisfying the assumptions of Theorem \ref{theorem} and $(f,\phi)$ the maximal solution to the modified massless relativistic Vlasov-Poisson system \eqref{VP1bis}-\eqref{VP2bis} such that $f(0,\cdot , \cdot ) = f_0$. Note that there exists $C>0$ such that\footnote{We refer to Appendix $B$ of \cite{massless} for the proof of a similar result which ensues from straightforward but tedious computations.} $\E_{N}^5[f](0) \leq C \epsilon$ and, considering possibly $\widetilde{\epsilon}=C \epsilon$, we can suppose without loss of generality that $C=1$. We fix for all the remaining of the proof $0 < \delta \leq \frac{1}{4}$ and we consider $T \geq 0$ the largest time such that
\begin{eqnarray}
\forall \hspace{0.5mm} t \in [0,T], \hspace{2.8cm} \E_{N-3}^5[f](t) & \leq & 3 \epsilon (1+t)^{\delta}, \label{boot1} \\ 
\forall \hspace{0.5mm} t \in [0,T], \hspace{3.15cm} \E_N^3[f](t) & \leq & 3 \epsilon (1+t)^{\delta}, \label{boot2} \\
\forall \hspace{0.5mm} t \in [0,T], \hspace{3.1cm} \E_N^2[f](t) & \leq & 3 \epsilon, \label{boot3} \\
\forall \hspace{0.5mm} |\beta| \leq N, \hspace{3mm} \forall \hspace{0.5mm} t \in [0,T], \hspace{1.2cm} \left\| \int_{\R^3_v} \left| \widehat{Z}^{\beta} f \right| dv \right\|_{L^2(\Sigma_t)} & \leq & C_{L^2} \frac{ \epsilon }{1+t}, \label{bootL2}
\end{eqnarray}
where $C_{L^2}>0$ is a sufficiently large constant which will be fixed in Section \ref{secL2}. In view of standard well-posedness arguments and the smallness assumptions on the particle density $f$, we have $T>0$. Theorem \ref{theorem} will then follow if we improve these boostrap assumptions, for $\epsilon$ small enough, independently of $T$. Using Proposition \ref{decay}, one immediately obtains from \eqref{boot1} and \eqref{boot3} that
\begin{eqnarray}\label{decayf}
 \forall \hspace{0.5mm} (t,x) \in [0,T] \times \R^3, \hspace{2mm} |\beta| \leq N-6, \hspace{1cm} \int_{v \in \R^3_v} z^3 |\widehat{Z}^{\beta} f|(t,x,v) dv & \lesssim & \epsilon \frac{(1+t)^{\delta}}{\tau_-^2 (1+r)^2}, \\ 
\forall \hspace{0.5mm} (t,x) \in [0,T] \times \R^3, \hspace{2mm} |\beta| \leq N-3, \hspace{1cm} \int_{v \in \R^3_v} |\widehat{Z}^{\beta} f|(t,x,v) dv & \lesssim &  \frac{\epsilon}{\tau_-^2 (1+r)^2} \label{decayf2}
\end{eqnarray}
The remaining of the proof is divided in three parts. First, we prove boundedness and pointwise decay estimates on the potential $\phi$, allowing us to prove that $f$ vanishes for all $|v| \leq 1$. Then, we improve the bootstrap assumptions \eqref{boot1}-\eqref{boot3} concerning the $L^1$ norms of the particle density. Finally, we improve the boostrap assumption \eqref{bootL2} on the $L^2$ norms of $f$.

\section{Estimates for the potential}\label{sec5phi}
The purpose of this section is to prove the following result.
\begin{Pro}\label{estiphi}
Let $Z^{\gamma} \in \Vv^{|\gamma|}$.
\begin{itemize}
\item If $|\gamma| \leq N$, then
$$ \forall \hspace{0.5mm} t \in [0,T[, \hspace{1cm} \left\| \nabla_x Z^{\gamma} \phi \right\|_{L^2(\Sigma_t)} \hspace{2mm} \lesssim \hspace{2mm} \frac{\epsilon}{\sqrt{1+t}}.$$
Moreover, if the total charge $Q_{Z^{\gamma} \phi}$ is equal to $0$, then
$$ \forall \hspace{0.5mm} t \in [0,T[, \hspace{1cm} \left\| \nabla_x Z^{\gamma} \phi \right\|_{L^2(\Sigma_t)}  \hspace{2mm} \lesssim  \hspace{2mm} \frac{\epsilon}{1+t}.$$
\item If $|\gamma| \leq N-3$, then
$$ \forall \hspace{0.5mm} (t,x)  \in [0,T[ \times \R^3, \hspace{1cm} \left| \nabla_x Z^{\gamma}  \phi \right|(t,x)  \hspace{2mm} \lesssim  \hspace{2mm} \frac{\epsilon}{(1+t)(1+|x|)}.$$
Moreover, if $Z^{\gamma} \phi$ is spherically symmetric, i.e. if $Z^{\gamma} \in \VV_{\mathbb{S}}^{|\gamma|}$, then
$$ \forall \hspace{0.5mm} (t,x)  \in [0,T[ \times \R^3, \hspace{1cm} \left| \nabla_x Z^{\gamma}  \phi \right|(t,x)  \hspace{2mm} \lesssim  \hspace{2mm} \frac{\epsilon}{(1+t+|x|)^2}.$$
\end{itemize}
\end{Pro}
\begin{Rq}
Recall from Proposition \ref{Procharge} that $Q_{Z^{\gamma} \phi} = (-1)^{|\gamma|} Q_{\phi}$ if $ Z^{\gamma} = S^{|\gamma|}$ and $ Q_{Z^{\gamma} \phi} = 0$ otherwise.
\end{Rq}
Proposition \ref{comu} gives us that, for all $|\gamma| \leq N$, there exists $C^{\gamma}_{\xi}$ such that
\begin{equation}\label{eq:lapla}
 \Delta Z^{\gamma} \phi \hspace{2mm} = \hspace{2mm} \int_{\R^3_v} g_{\gamma} dv  , \hspace{1cm} \text{with} \hspace{1cm} g_{\gamma} \hspace{2mm} := \hspace{2mm} \sum_{|\xi| \leq |\gamma|} C^{\gamma}_{\xi}  \widehat{Z}^{\xi} f .
 \end{equation}
Hence, by the bootstrap assumptions \eqref{boot3} and \eqref{bootL2} as well as the pointwise decay estimate \eqref{decayf2}, we have
\begin{eqnarray}
\forall \hspace{0.5mm} t \in [0,T[, \hspace{1cm} \int_{\Sigma_t} \int_{\R^3_v} z^2 |g_{\gamma}| dv dx & \lesssim & \epsilon \label{gL1} \\ 
\forall \hspace{0.5mm} t \in [0,T[, \hspace{1cm} \left\| \int_{\R^3_v} |g_{\gamma}| dv \right\|_{L^2(\Sigma_t)} & \lesssim & \frac{\epsilon}{1+t} \label{gL2} \\ 
\forall \hspace{0.5mm} (t,x) \in [0,T[ \times \R^3, \hspace{1cm}  \int_{\R^3_v}  |g_{\gamma}|(t,x,v) dv  & \lesssim & \frac{\epsilon}{(1+t+|x|)^2}, \hspace{1cm} \text{if} \hspace{2mm} |\gamma| \leq N-3 \label{gL8} .
\end{eqnarray}

\subsection{The spherically symmetric derivatives}

We derive here global bounds and pointwise decay estimates for the cases where $Z^{\gamma} \in \Vv_{\Sp}^{|\gamma|}$, so that, according to Proposition \ref{Prospheri}, $Z^{\gamma} \phi$ is spherically symmetric and \eqref{eq:lapla} then becomes
\begin{equation}\label{eq:laplar} 
\Delta Z^{\gamma} \phi \hspace{2mm} = \hspace{2mm} \frac{1}{r^2} \partial_r \left( r^2 \partial_r Z^{\gamma} \phi \right) \hspace{2mm} = \hspace{2mm} \int_{\R^3_v} g_{\gamma} dv.
\end{equation}

\begin{Pro}\label{phi0}
If $Z^{\gamma} \in \Vv_{\Sp}^{|\gamma|}$ and $|\gamma| \leq N-3$, we have
$$ \forall \hspace{0.5mm} t \in [0,T[ \times \R^3,  \hspace{1.5cm} |\nabla_x Z^{\gamma} \phi|(t,x) \hspace{2mm} \lesssim \hspace{2mm} \frac{\epsilon}{(1+t+|x|)^2}.$$
Moreover, this estimate also holds for the higher order derivatives $N-2 \leq |\gamma| \leq N$ in the region $|x| \geq 1$.
\end{Pro}
\begin{proof}
Let $|\gamma| \leq N$. Start by integrating \eqref{eq:laplar} between $0$ and $r$ in order to get
\begin{equation}\label{eq:laplar1}
 r^2 \partial_r Z^{\gamma} \phi \hspace{2mm} = \hspace{2mm} \int_{\rho=0}^r \int_{\R^3_v} g_{\gamma} dv \rho^2 d\rho.
 \end{equation}
Let us consider several cases.
\begin{itemize}
\item If $r \leq 1$ and $|\gamma| \leq N-3$, \eqref{gL8} gives
$$ \left| \partial_r Z^{\gamma} \phi \right|(t,r) \hspace{2mm} \lesssim \hspace{2mm} \frac{1}{r^2} \int_0^r \frac{\epsilon}{(1+t)^2} \rho^2 d \rho \hspace{2mm} \lesssim \hspace{2mm} \frac{\epsilon r^3}{(1+t)^2 r^2} \hspace{2mm} \leq \hspace{2mm} \frac{\epsilon}{(1+t)^2} .$$
\item If $r \geq 1$ and $t \geq 2r$, we get by Lemma \ref{zprop} and \eqref{gL1},
\begin{eqnarray}
\nonumber \left| \partial_r Z^{\gamma} \phi \right|(t,r) & \lesssim & \frac{1}{r^2} \int_0^r \int_{\R^3_v} |g_{\gamma}| dv \rho^2 d \rho \hspace{2mm} \lesssim \hspace{2mm} \int_0^r \int_{\R^3_v} \frac{z^2}{1+(t-\rho)^2} |g_{\gamma}| dv \rho^2 d \rho \\ \nonumber
 & \lesssim & \frac{1}{1+t^2} \int_0^r \int_{\R^3_v} z^2|g_{\gamma}| dv \rho^2 d \rho \hspace{2mm} \lesssim \hspace{2mm} \frac{1}{(1+t)^2} \int_{\Sigma_t} \int_{\R^3_v} z^2|g_{\gamma}| dv dx \hspace{2mm} \lesssim \hspace{2mm} \frac{\epsilon}{(1+t)^2}.
 \end{eqnarray}
\item Otherwise, $r \geq 1$ and $r \geq \frac{t}{2}$, so that $4r^2 \geq (1+r)^2$. One then gets, using again \eqref{gL1},
$$ \left| \partial_r Z^{\gamma} \phi \right|(t,r) \hspace{2mm} \lesssim \hspace{2mm} \frac{1}{r^2} \int_0^r \int_{\R^3_v} |g_{\gamma}| dv \rho^2 d \rho \hspace{2mm} \lesssim \hspace{2mm} \frac{1}{(1+r)^2} \int_{\Sigma_t} \int_{\R^3_v} |g_{\gamma}| dv  dx \hspace{2mm} \lesssim \hspace{2mm} \frac{\epsilon}{(1+r)^2}.$$
\end{itemize}
\end{proof}
We now prove the following $L^2$ bounds.
\begin{Pro}\label{phil}
Let $Z^{\gamma} \in \Vv^{|\gamma|}_{\Sp}$ with $|\gamma| \leq N$. Then,
\begin{eqnarray}
\nonumber \forall \hspace{0.5mm} t \in [0,T[, \hspace{1cm}  \left\| \nabla_x Z^{\gamma} \phi \right\|_{L^2(\Sigma_t)} & \lesssim & \frac{\epsilon}{1+t} \hspace{1.4cm} \text{if} \hspace{2mm} Q_{Z^{\gamma} \phi}=0, \\ \nonumber
\forall \hspace{0.5mm} t \in [0,T[, \hspace{1cm}  \left\| \nabla_x Z^{\gamma} \phi \right\|_{L^2(\Sigma_t)} & \lesssim & \frac{\epsilon}{\sqrt{1+t}} \hspace{1cm} \text{otherwise}.  
\end{eqnarray}
\end{Pro}
\begin{proof}
We fix $Z^{\gamma} \in \Vv^{|\gamma|}_{\Sp}$ with $|\gamma| \leq N$ and we first consider the region $r \leq 1$. Multiplying \eqref{eq:laplar} by $r^4 \partial_r Z^{\gamma} \phi $ and integrating between $0$ and $r$, we get
\begin{equation}\label{eq:proofL2phi}
 r^4 |\partial_r Z^{\gamma} \phi |^2 \hspace{2mm} = \hspace{2mm} \int_{\rho=0}^r \int_{\R^3_v} g_{\gamma} dv \cdot  \partial_r Z^{\gamma} \phi \cdot \rho^4 d\rho \hspace{2mm} \leq \hspace{2mm} \Big| \int_{\rho=0}^r \Big| \rho^2 \int_{\R^3_v} g_{\gamma} dv \Big|^2 \rho^2 d\rho \int_{\rho=0}^r  |\partial_r Z^{\gamma} \phi |^2 \rho^2 d \rho \Big|^{\frac{1}{2}} .
 \end{equation}
Integrating in the variable $r$ between $0$ and $1$, we obtain
$$ \left\| \partial_r Z^{\gamma} \phi \right\|_{L^2(|x| \leq 1)}^2 \hspace{2mm} \lesssim \hspace{2mm} \int_{r=0}^1  \left\| \frac{\rho^2}{r^2} \int_{\R^3_v} g_{\gamma} dv \right\|_{L^2(\rho \leq r)} \hspace{-0.3mm} \left\| \partial_r Z^{\gamma} \phi \right\|_{L^2(\rho \leq r)} dr \hspace{2mm} \lesssim \hspace{2mm} \left\| \int_{\R^3_v}  g_{\gamma} dv \right\|_{L^2(\Sigma_t)} \hspace{-0.3mm} \left\| \partial_r Z^{\gamma} \phi \right\|_{L^2(|x| \leq 1)} \hspace{-0.3mm} .$$
Using \eqref{gL2}, we finally get
\begin{equation}\label{regionrsmall}
 \left\| \partial_r Z^{\gamma} \phi \right\|_{L^2(|x| \leq 1)} \hspace{2mm} \lesssim \hspace{2mm} \frac{\epsilon}{1+t}.
 \end{equation}
Now, recall from Proposition \ref{phi0} that $ |\partial_r Z^{\gamma} \phi |(t,x) \lesssim \epsilon (1+t+|x|)^{-2}$ for all $|x| \geq 1$. Hence,
$$ \left\| \partial_r Z^{\gamma} \phi \right\|_{L^2(|x| \geq 1)}^2 \hspace{2mm} \lesssim \hspace{2mm} \epsilon \int_{r=1}^{+ \infty} \frac{  r^2 }{(1+t+r)^4} dr \hspace{2mm} \lesssim \hspace{2mm} \epsilon  \int_{r=0}^{+ \infty} \frac{ dr }{(1+t+r)^2}  \hspace{2mm} \lesssim \hspace{2mm} \frac{\epsilon}{1+t}.$$
This implies the result in the general case, which will be sufficient for us to improve all the bootstrap assumptions. Let us now improve the decay rate for the chargeless derivatives of the potential, i.e. we suppose that $Q_{Z^{\gamma} \phi} =0$. According to \eqref{regionrsmall}, we only need to consider the regions $1 \leq r \leq t$ and $r \geq \max(1,t)$, so that $ |Z^{\gamma} \phi |(t,x) \lesssim \epsilon (1+t+|x|)^{-2}$. By \eqref{eq:laplar1} and using $\tau_- \leq z$ (see Lemma \ref{zprop}), we have
$$\forall \hspace{0.5mm} r \geq 1, \hspace{1cm} r^2 | \partial_r Z^{\gamma} \phi |^2 \hspace{2mm} \leq \hspace{2mm} \int_{\rho=0}^r \int_{\R^3_v} g_{\gamma} dv \rho^2 d \rho \cdot \partial_r Z^{\gamma} \phi \hspace{2mm} \lesssim \hspace{2mm} \frac{\epsilon}{(1+t)^2} \int_0^r \int_{\R^3_v} \frac{z^2}{(1+|t-\rho|)^2} g_{\gamma} dv \rho^2 d \rho .$$
Integrating the last inequality between $r=1$ and $t$, we obtain, using \eqref{gL1},
\begin{equation}\label{L2phichargeless}
\hspace{-1mm} \int_{r=1}^t  | \partial_r Z^{\gamma} \phi |^2 r^2 dr \hspace{2mm} \lesssim \hspace{2mm} \frac{\epsilon}{(1+t)^2} \int_{r=1}^t \frac{1}{(1+t-r)^2} \int_0^r \int_{\R^3_v} z^2 g_{\gamma} dv \rho^2 d \rho dr \hspace{2mm} \lesssim \hspace{2mm} \frac{\epsilon^2}{(1+t)^2}  \int_{r=0}^t \frac{dr}{(1+t-r)^2} .
\end{equation}
Finally, note that, as $Z^{\gamma} \phi$ is chargeless and spherically symmetric,
$$\lim_{r \rightarrow +\infty} r^2 \partial_r Z^{\gamma} \phi \hspace{2mm} = \hspace{2mm} \frac{1}{4 \pi} \lim_{r \rightarrow +\infty} \int_{\Sp_{t,r}} \partial_r Z^{\gamma} \phi d \Sp_{t,r} \hspace{2mm} = \hspace{2mm} Q_{Z^{\gamma} \phi} \hspace{2mm} = \hspace{2mm} 0.$$  Hence, multiplying \eqref{eq:laplar} by $r^2$ and integrating between $r \geq t$ and $+ \infty$, we get
$$ r^2 |\partial_r Z^{\gamma} \phi |^2 \hspace{2mm} = \hspace{2mm} -\int_{\rho = r}^{+\infty} \int_{\R^3_v} g_{\gamma} dv \rho^2 d \rho \cdot \partial_r Z^{\gamma} \phi \hspace{2mm} \lesssim \hspace{2mm} \frac{\epsilon}{(1+t)^2} \int_{\rho = r}^{+\infty} \int_{\R^3_v} \frac{z^2}{(1+|t-\rho|)^2} |g_{\gamma}| dv \rho^2 d \rho .$$
Integrating the last inequality between $r=t$ and $+\infty$, one obtains, using \eqref{gL1},
$$ \int_{r=t}^{+\infty}  | \partial_r Z^{\gamma} \phi |^2 r^2 dr \hspace{1.9mm} \lesssim \hspace{1.9mm} \frac{\epsilon}{(1+t)^2} \int_{r=t}^{+\infty} \frac{1}{(1+r-t)^2} \int_r^{+ \infty} \int_{\R^3_v} z^2 g_{\gamma} dv \rho^2 d \rho dr \hspace{1.9mm} \lesssim \hspace{1.9mm} \frac{\epsilon^2}{(1+t)^2}  \int_{r=t}^{+\infty} \frac{dr}{(1+r-t)^2} .$$
Then, $\left\| \nabla_x Z^{\gamma} \phi \right\|_{L^2(\Sigma_t)}\lesssim \epsilon (1+t)^{-1}$ follows from the last estimate, \eqref{L2phichargeless} and \eqref{regionrsmall}.
\end{proof} 

\subsection{The other derivatives}

If there is at least one space translation composing $Z^{\gamma}$, the function $Z^{\gamma} \phi$ is not spherically symmetric and we cannot make the same computations as in the previous subsection in order to prove estimates on it. Instead, we take advantage of the fact that for $\psi$ solution to 
$$ \Delta \psi \hspace{2mm} = \hspace{2mm} F,$$
with $F$ a sufficiently regular function, $\nabla_x \nabla_x \psi$ has a better behavior than $\nabla_x \psi$. More precisely, applying the Calder\'on-Zygmund equality and then using the estimate \eqref{gL2}, one has
\begin{equation}\label{calde}
\forall \hspace{0.5mm} |\alpha| \leq N, \hspace{1cm} \left\| \nabla_x \nabla_x Z^{\alpha} \phi \right\|_{L^2(\Sigma_t)} \hspace{2mm} = \hspace{2mm} \left\| \int_{\R^3_v} g_{\alpha} dv \right\|_{L^2(\Sigma_t)} \hspace{2mm} \lesssim \hspace{2mm} \frac{\epsilon}{1+t},
\end{equation}
which will allow us to prove the following result.
\begin{Pro}\label{estinosph}
Let $Z^{\gamma} \in \Vv^{|\gamma|}$, with $1 \leq |\gamma| \leq N$ and which is composed of at least one space translation $\partial_i$, $i \in \{1,2,3\}$. Then,
\begin{eqnarray}
\forall \hspace{0.5mm} t \in [0,T[, \hspace{1cm} \left\| \nabla_x Z^{\gamma} \phi \right\|_{L^2(\Sigma_t)} & \lesssim & \frac{\epsilon}{1+t} , \label{energyphi2} \\
\forall \hspace{0.5mm} (t,x) \in [0,T[ \times \R^3, \hspace{1cm} \left| \nabla_x Z^{\gamma} \phi \right| (t,x) & \lesssim & \frac{\epsilon}{(1+t)(1+r)} \hspace{1cm} \text{if} \hspace{2mm} |\gamma| \leq N-2.
\end{eqnarray}
\end{Pro}
\begin{proof}
Note first that, for all $i \in \{1,2,3\}$ and $Z \in \Vv$,
 \begin{equation}\label{comuonetra}
 [Z, \partial_i]=0 \hspace{1cm} \text{or} \hspace{1cm} \exists \hspace{0.5mm} k \in \{1,2,3\}, \hspace{3mm} [Z, \partial_i]= \pm \partial_k.
 \end{equation}
Consequently, if $Z^{\gamma}$ satisfies the hypotheses of the proposition, we have
\begin{equation}\label{eq:wlog}
 \left| \nabla_x Z^{\gamma} \phi \right| \hspace{2mm} \lesssim \hspace{2mm} \sum_{|\alpha| \leq |\gamma| -1} \left| \nabla_x \nabla_x Z^{\alpha} \phi \right|
\end{equation}
and the $L^2$ estimate is then implied by \eqref{calde}. We now focus on the pointwise decay estimates and we assume, without loss of generality (in view of \eqref{eq:wlog}), that $Z^{\gamma}=\partial_k Z^{\alpha}$, with $|\alpha| \leq N-3$ and $k \in \{1,2,3 \}$. Fix also $(t,x)=(t,|x| \omega) \in [0,T[ \times \R^3$. Applying a standard $L^2-L^{\infty}$ Sobolev inequality and using the energy bound \eqref{energyphi2}, we get
$$ |\nabla_x \nabla_x Z^{\alpha} \phi |(t,x) \hspace{2mm} \lesssim \hspace{2mm} \sum_{|\beta| \leq 2 } \left\| \nabla_x^{\beta} \nabla_x \nabla_x Z^{\alpha} \phi \right\|_{L^2(\Sigma_t)} \hspace{2mm} \lesssim \hspace{2mm} \sum_{|\kappa| \leq N-1 } \left\|\nabla_x \nabla_x Z^{\kappa} \phi \right\|_{L^2(\Sigma_t)}  \hspace{2mm} \lesssim \hspace{2mm} \frac{\epsilon}{1+t},$$
which gives us the expected decay rate in the region $|x| \leq 1$. Otherwise, we have
\begin{eqnarray}
\nonumber | \nabla_x \nabla_x Z^{\alpha} \phi |^2 (t,r \omega) & = & - 2 \int_{\rho=|x|}^{+ \infty} \partial_r \nabla_x \nabla_x Z^{\alpha} \phi \cdot \nabla_x \nabla_x Z^{\alpha} \phi  d \rho \\ \nonumber
& \leq & 2 \int_{\rho=|x|}^{+ \infty} \left| \partial_r \nabla_x \nabla_x Z^{\alpha} \phi \right| \cdot \left| \nabla_x \nabla_x Z^{\alpha} \phi \right| \frac{\rho^2}{|x|^2}  d \rho \\ \nonumber
& \lesssim & \frac{2}{|x|^2} \int_{\rho=0}^{+ \infty} \left(  \left| \nabla_x \nabla_x \nabla_x Z^{\alpha} \phi \right|^2+ \left| \nabla_x \nabla_x Z^{\alpha} \phi \right|^2 \right) \rho^2  d \rho.
\end{eqnarray}
Now, apply to $\omega \mapsto \nabla_x \nabla_x Z^{\alpha} \phi$ and $\omega \mapsto \partial_r \nabla_x \nabla_x Z^{\alpha} \phi$ the following $L^2$ Sobolev inequality on the unit sphere $\mathbb{S}^2$, 
$$ \left\| v \right\|_{L^{\infty}(\Sp^2)} \hspace{2mm} \lesssim \hspace{2mm} \left\| v \right\|_{L^{2}(\Sp^2)}+\left\| \slashed{\nabla} v \right\|_{L^{2}(\Sp^2)}+\left\| \slashed{\nabla} \slashed{\nabla} v \right\|_{L^{2}(\Sp^2)}.$$
Then, note that
\begin{itemize}
\item $r\slashed{\nabla}=(\partial_{\omega_1}, \partial_{\omega_2})$ on the sphere $|x|=r$ and the angular derivative $\partial_{\omega_i}$, for $i \in \{1,2 \}$, satisfies
$$ \partial_{\omega_i} \hspace{2mm} = \hspace{2mm} \sum_{k \leq 1 < l \leq 3} C^{k,l}(\omega) \Omega_{kl},$$
where $C^{k,l}$ are bounded functions on the sphere $\mathbb{S}^2$.
\item $[\Omega_{ij}, \partial_k] = \delta^i_k \partial_j-\delta^j_k \partial_i$ for all $1 \leq i < j \leq 3$ and $1 \leq k \leq 3$.
\end{itemize}
Thus, we deduce that
\begin{eqnarray}
\nonumber | \nabla_x \nabla_x Z^{\alpha} \phi |^2 (t,|x| \omega) & \lesssim & \frac{1}{|x|^2}\sum_{|\kappa| \leq N} \int_{\rho=0}^{+ \infty} \int_{\overline{\omega} \in \mathbb{S}^2} \left| \nabla_x \nabla_x Z^{\kappa} \phi \right|^2(t, \rho \overline{\omega})  d \overline{\omega} \rho^2 d \rho \\ \nonumber
& \leq &  \frac{1}{|x|^2} \sum_{|\kappa| \leq N} \left\| \nabla_x \nabla_x Z^{\kappa} \phi \right\|_{L^2(\Sigma_t)}^2 \\ \nonumber
& \leq & \frac{\epsilon}{(1+t)^2(1+|x|)^2} \hspace{8mm} \text{by \eqref{calde} and since $|x|+1 \leq 2|x|$}.
\end{eqnarray}
This concludes the proof.
\end{proof}

\subsection{The Vlasov field vanishes for small velocities}\label{Secvanish}

The purpose of this section is to prove that 
\begin{equation}\label{supportvanish}
 |v| \leq 1 \hspace{2mm} \Rightarrow \hspace{2mm} \forall \hspace{0.5mm} (t,x) \in [0,T[ \times \R^3, \hspace{2mm} f(t,x,v)=0.
\end{equation}
This implies in particular the following result.
\begin{Pro}\label{vanishsmallv}
The particle density $f$ vanishes for all $|v| \leq 1$, so that $\T^{\chi}_{\phi}(f)=\T_{\phi}(f)$ on $[0,T[ \times \R^3_x \times \R^3_v$. Consequently, $(f,\phi)$ is a solution to the massless relativistic Vlasov-Poisson system \eqref{VP1}-\eqref{VP2} on $[0,T[$.
 \end{Pro}

Let $x \in \mathbb{R}^n$, $|v| \geq 2$ and $(X,V)$ be the characteristic of the transport operator $\T^{\chi}_{\phi}$ satisfying $(X(0),V(0))=(x,v)$. We have in particular
$$ \hspace{3mm} \frac{dV}{ds}(s)=\nabla_x \phi (s,X(s)),$$
which implies, using the pointwise decay estimate on $\phi$ given by Proposition \ref{estiphi},
$$\left| \frac{d(|V|)}{ds} \right|(s) \hspace{2mm} = \hspace{2mm} 2 \left| \left<\nabla_x \phi (s,X(s)),\frac{V}{|V|}(s) \right> \right| \hspace{2mm} \leq \hspace{2mm} 2 \left| \nabla_x \phi \right| (s,X(s)) \hspace{2mm} \lesssim \hspace{2mm} \frac{\epsilon}{(1+s)^2}.$$
Consequently, there exists $C_0 >0$ independent of $v$ such that
$$ \forall \hspace{0.5mm} t \in [0,T[, \hspace{5mm} |V|(t)=|v|+\int_0^t \frac{d(|V|)}{ds}(s) ds \hspace{2mm} \geq \hspace{2mm} 2-\int_0^t \frac{C_0 \epsilon}{(1+s)^2} ds \hspace{2mm} \geq \hspace{2mm} 2-\epsilon C_0 \int_0^{+\infty} \frac{ds}{(1+s)^2} .$$
By taking $\epsilon$ sufficiently small, we obtain $\inf_{[0,T[} |V| \geq 1$. In view of \eqref{charac} and that $f(0,\cdot , w)=0$ for all $|w| \leq 2$, we finally deduce that \eqref{supportvanish} holds.

\section{Improvement of the boostrap assumptions \eqref{boot1}, \eqref{boot2} and \eqref{boot3}}\label{Secbootf}

Recall from Proposition \ref{energyparti} that, for $p \in \mathbb{N}$ and $|\xi| \leq N$,
\begin{eqnarray}
\nonumber \hspace{-6mm} \left\| z^p \widehat{Z}^{\beta} f (t, \cdot , \cdot) \right\|_{L^1_{x,v}} \hspace{-1mm}  & \leq & \left\| z^p \widehat{Z}^{\beta} f (0, \cdot , \cdot) \right\|_{L^1_{x,v}} + \int_0^t \int_{\Sigma_s} \int_{\R^3_v} \left| \T_{\phi}(z^p \ZZ^{\beta} f ) \right| \frac{dv}{v^0} dx ds \\ 
& \leq & \left\| z^p \widehat{Z}^{\beta} f (0, \cdot , \cdot) \right\|_{L^1_{x,v}} + \int_0^t \int_{\Sigma_s} \int_{\R^3_v} \left( z^p \left| \T_{\phi}( \ZZ^{\beta} f ) \right| + pz^{p-1} \left|  \T_{\phi} (z)  \ZZ^{\beta} f  \right| \right) dv dx ds , \label{eq:machin1}
\end{eqnarray}
since $\frac{1}{v^0} \leq 1$ on the support of $f$ (see Proposition \ref{vanishsmallv}). Using the crude inequality\footnote{The weight $r$ has to be transformed since otherwise we would be led to deal with $r^2 |\nabla_x Z^{\gamma} \phi|$, which is not uniformly bounded in $r$ if $Z^{\gamma}$ is composed of at least one translation in space. We use the bad inequality $\tau_-+t \leq z(1+t)$ in order to unify the estimate of \eqref{machin2} with the one of \eqref{eq:machin1}.} $r \leq \tau_-+t \leq z(1+t)$, which ensues from Lemma \ref{zprop}, as well as $\frac{1}{v^0} \leq 1$ on the support of $f$, we have
\begin{equation}\label{machin2}
 \int_{\Sigma_t} \int_{\R^3_v} r z^{p-1} \left| \widehat{Z}^{\beta} \left( \T_{\phi} (\widehat{Z}^{\alpha} f) \right) \right| \frac{dv}{v^0} dx \hspace{2mm} \lesssim \hspace{2mm} (1+t)\int_{\Sigma_t} \int_{\R^3_v} z^p \left| \widehat{Z}^{\beta} \left( \T_{\phi} (\widehat{Z}^{\alpha} f) \right) \right| dv dx.
 \end{equation} As $\T(z)=0$ by Lemma \ref{zprop} and according to commutation formula of Proposition \ref{comu}, we then have, for $Q \leq N$,
\begin{eqnarray}
\nonumber \E_Q^p[f] (t)-\E_Q^p[f] (0) \hspace{-1.3mm} & \lesssim & \hspace{-1.3mm} \int_0^t \hspace{-0.2mm} \int_{\Sigma_s} \hspace{-0.2mm} \int_{\R^3_v} \hspace{-0.2mm} \Big( \hspace{-0.2mm} \sum_{\begin{subarray}{} |\gamma|+|\kappa| \leq Q \\ \hspace{1mm} |\kappa| \leq Q-1 \end{subarray} } \hspace{-0.3mm} z^p \left| \nabla_x Z^{\gamma} \phi \cdot \nabla_v  \ZZ^{\kappa} f  \right| + \hspace{-0.3mm} \sum_{|\beta| \leq Q}  \left| \nabla_x \phi \cdot \nabla_v z \right| z^{p-1} \left|  \ZZ^{\beta} f  \right| \Big) v^0 dv dx ds \\ 
& & \hspace{-1.3mm} + \sum_{\begin{subarray}{} |\gamma|+|\kappa| \leq Q \\ \hspace{1mm} |\kappa| \leq Q-1 \end{subarray} } (1+t) \int_{\Sigma_t}  \int_{\R^3_v}  z^p \left| \nabla_x Z^{\gamma} \phi \cdot \nabla_v  \ZZ^{\kappa} f  \right|v^0 dv dx. \label{energyfff}
\end{eqnarray}
The following lemma will then be useful.
\begin{Lem}\label{reducing}
Let $Z^{\gamma} \in \Vv^{|\gamma|}$ and $g : [0,T[ \times \R^3_x \times \R^3_v \rightarrow \R$ be a sufficiently regular function. Then,
$$   v^0\left| \nabla_x Z^{\gamma} \phi \cdot \nabla_v  g  \right| \hspace{2mm} \lesssim \hspace{2mm} z \frac{r}{1+t+r} \left| \nabla_x Z^{\gamma} \phi \right| |  \nabla_x g  |+ \sum_{|\alpha| \leq |\gamma| } \sum_{\ZZ \in \VV}  \left| \nabla_x Z^{\alpha} \phi \right| |   \ZZ g  |.$$
\end{Lem}
\begin{proof}
Start by expanding the scalar product $\nabla_x Z^{\gamma} \phi \cdot \nabla_v  g$ in the spherical frame in $v$ $(\partial_{|v|}, \frac{1}{|v|} \partial_{\theta^1_v}, \frac{1}{|v|} \partial_{\theta^2_v})$, so that
$$ v^0 \nabla_x Z^{\gamma} \phi \cdot \nabla_v  g \hspace{2mm} = \hspace{2mm} v^0 \frac{v^i}{|v|} \partial_i Z^{\gamma} \phi \cdot \frac{v^i}{|v|} \partial_{v^i} g+ \sum_{i=1}^2 \frac{v^0}{|v|^2} \left( \nabla_x Z^{\gamma} \phi \right)^{\theta^i_v} \left( \nabla_v  g \right)^{\theta^i_v}.$$
Note now that
$$ \left| v^0 \frac{v^i}{|v|} \partial_i Z^{\gamma} \phi \cdot \frac{v^i}{|v|} \partial_{v^i} g \right| \hspace{2mm} \leq \hspace{2mm} \left|  \nabla_x Z^{\gamma} \phi \right| \cdot \left| v^i \partial_{v^i} g \right|.$$
Since $v^i \partial_{v^i} \in \VV$, the right hand side of the previous inequality has the requested form. To deal with the remaining term, we use that
$$ \frac{1}{|v|} \partial_{\theta^i_v} \hspace{2mm} = \hspace{2mm} \sum_{1 \leq k < l \leq 3} C^i_{k,l}(\theta^1_v, \theta^2_v) \Big( \frac{v^k}{v^0} \partial_{v^l}-\frac{v^l}{v^0} \partial_{v^k} \Big),$$
where $C^i_{k,l}$ are bounded functions (on the sphere $|v|=1$). Consequently,
\begin{eqnarray}
\nonumber \frac{v^0}{|v|^2} \left| \left( \nabla_x Z^{\gamma} \phi \right)^{\theta^i_v} \left( \nabla_v  g \right)^{\theta^i_v} \right| & \lesssim & v^0 \sum_{1 \leq k < l \leq 3} \left| \frac{v^k}{v^0} \partial_l Z^{\gamma} \phi- \frac{v^l}{v^0} \partial_k Z^{\gamma} \phi \right| \sum_{1 \leq k < l \leq 3} \left| \frac{v^k}{v^0} \partial_{v^l} g- \frac{v^l}{v^0} \partial_{v^k} g \right| \\ \nonumber
& \lesssim & \sum_{1 \leq k < l \leq 3} \left| \frac{v^k}{v^0} \partial_l Z^{\gamma} \phi- \frac{v^l}{v^0} \partial_k Z^{\gamma} \phi \right| \sum_{1 \leq k < l \leq 3} \left( \left|\widehat{\Omega}_{kl} g \right|+\left| x^k \partial_{l} g- x^l \partial_{k} g \right| \right).
\\ \nonumber
& \lesssim & \sum_{1 \leq k < l \leq 3} \left| \nabla_x Z^{\gamma} \phi \right|  \left| \widehat{\Omega}_{kl} g \right|+  r\sum_{1 \leq k < l \leq 3} \left| \frac{v^k}{v^0} \partial_l Z^{\gamma} \phi- \frac{v^l}{v^0} \partial_k Z^{\gamma} \phi \right| \left| \nabla_x g \right| .
\end{eqnarray}
The first term on the right hand side of the last inequality has the requested form. For the second one, apply Lemma \ref{null} in order to get
\begin{equation}\label{eq:number1}
r\sum_{1 \leq k < l \leq 3} \left| \frac{v^k}{v^0} \partial_l Z^{\gamma} \phi- \frac{v^l}{v^0} \partial_k Z^{\gamma} \phi \right| \left| \nabla_x g \right| \hspace{2mm} \lesssim \hspace{2mm} \frac{r z}{1+t+r}  \left| \nabla_x Z^{\gamma} \phi \right| \left| \nabla_x g\right|+\sum_{1 \leq k < l \leq 3} \frac{r}{1+t+r} \left| \Omega_{kl} Z^{\gamma} \phi \right| \left| \nabla_x g\right|.
\end{equation}
Again, the first term on the right hand side of the previous inequality has the requested form. To deal with the last one, remark that
\begin{itemize}
\item if $Z^{\gamma} \in \Vv_{\Sp}^{|\gamma|}$, which means that there is no space translation in $Z^{\gamma}$, then by Proposition \ref{Prospheri}, $Z^{\gamma} \phi$ is spherically symmetric and 
$$\sum_{1 \leq k < l \leq 3} \left| \Omega_{kl} Z^{\gamma} \phi \right| =0.$$
\item Otherwise, $\Omega_{kl} Z^{\gamma}$ is composed of at least one space translation. Using the commutation properties between $\partial_k$ and the vector fields of $\Vv$ (see \eqref{comuonetra}), we have
$$ \sum_{1 \leq k < l \leq 3} \left| \Omega_{kl} Z^{\gamma} \phi \right| \left| \nabla_x g \right| \hspace{2mm} \lesssim \hspace{2mm} \sum_{|\alpha| \leq |\gamma|} \left| \nabla_x  Z^{\alpha} \phi \right| \left| \nabla_x g \right|,$$
which concludes the proof.
\end{itemize}
\end{proof}
We then deduce the following result, which, in view of \eqref{energyfff}, will clearly be useful in order to improve the energy estimates on the Vlasov field.
\begin{Cor}\label{Corz}
Let $ p \in \{2,3,5\}$ and $\ZZ^{\beta} \in \VV^{|\beta|}$ such that $|\beta| \leq N-3$ if $p=5$ and $|\beta| \leq N$ otherwise. Then,
$$\int_0^t \int_{\Sigma_s} \int_{\R^3_v} v^0\left| \nabla_x \phi \cdot \nabla_v z \right| z^{p-1} \left|  \ZZ^{\beta} f  \right| dv dx ds \hspace{2mm} \lesssim \hspace{2mm} \epsilon^2.$$
\end{Cor}
\begin{proof}
Applying Lemma \ref{reducing} with $|\gamma|=0$ and $g=z$ gives us
\begin{eqnarray}
\nonumber \int_0^t \hspace{-0.2mm} \int_{\Sigma_s} \int_{\R^3_v} v^0\left| \nabla_x \phi \cdot \nabla_v z \right| z^{p-1} \left|  \ZZ^{\beta} f  \right|  dv dx ds \hspace{-1.2mm} & \lesssim & \hspace{-1.2mm} \int_0^t \hspace{-0.3mm} \int_{\Sigma_s} \left| \nabla_x \phi \right| \int_{\R^3_v} \hspace{-0.2mm} z^{p-1} \hspace{-0.3mm} \left( z |\nabla_x z |+ \sum_{\ZZ \in \VV} \left|  \ZZ(z)  \right| \right) \hspace{-0.3mm} \left| \ZZ^{\beta} f \right| dv dx ds \\ \nonumber
& \lesssim & \hspace{-1.2mm} \int_0^t \left\| \nabla_x \phi \right\|_{L^{\infty}(\Sigma_s)} \left\| \int_{\R^3_v} z^p\left| \ZZ^{\beta} f \right| dv \right\|_{L^1(\Sigma_s)} ds ,
\end{eqnarray}
since, by Lemma \ref{zprop}, $|\nabla_x z| \lesssim 1$ and $\sum_{\ZZ \in \VV} | \ZZ(z) | \lesssim z$. According to Proposition \ref{estiphi} and the boostrap assumptions \eqref{boot1}-\eqref{boot3}, we have 
$$\left\| \nabla_x \phi \right\|_{L^{\infty}(\Sigma_s)} \lesssim \epsilon (1+s)^{-2} \hspace{1cm} \text{and} \hspace{1cm} \left\| \int_{\R^3_v} z^p \left| \widehat{Z}^{\beta} f \right| dv \right\|_{L^1\Sigma_s)} \hspace{2mm} \lesssim \hspace{2mm} \epsilon (1+s)^{\delta},$$ so that, as $\delta \leq \frac{1}{4}$,
$$ \int_0^t \int_{\Sigma_s} \int_{\R^3_v} v^0\left| \nabla_x \phi \cdot \nabla_v z \right| z^{p-1} \left|  \ZZ^{\beta} f  \right|  dv dx ds \hspace{2mm} \lesssim \hspace{2mm} \int_0^t \frac{\epsilon}{(1+t)^2} \epsilon (1+t)^{\delta} ds \hspace{2mm} \lesssim \hspace{2mm} \epsilon^2 \int_0^{+ \infty} \frac{ds}{(1+s)^{\frac{3}{2}}} \hspace{2mm} \lesssim \hspace{2mm} \epsilon^2.$$
\end{proof}
Motivated by \eqref{energyfff} and Lemma \ref{reducing}, let us introduce, for $p \in \mathbb{N}$ and multi-indices $\gamma$, $\kappa$, the following integrals
$$ \mathfrak{I}^p_{\gamma,\kappa}(s) \hspace{2mm} := \hspace{2mm} \int_{\Sigma_s}\int_{\R^3_v} \left( z\frac{r}{1+t+r}+1 \right) \left| \nabla_x Z^{\gamma} \phi \right|  z^p |  \ZZ^{\kappa} f  | dv dx .$$
The remaining of this section will be devoted to the proof of the following propositions.
\begin{Pro}\label{pro1}
Let multi-indices $\gamma$ and $\kappa$ such that 
$$|\gamma|+|\kappa| \leq N-2, \hspace{10mm} |\gamma| \leq N-3 \hspace{6mm} \text{and} \hspace{6mm} |\kappa| \leq N-3.$$
Then, for all $t \in [0,T[$,
$$\mathfrak{I}^5_{\gamma,\kappa}(t) \hspace{2mm} \lesssim \hspace{2mm} \frac{\epsilon^2}{ (1+t)^{1-\delta}}.$$
\end{Pro}
\begin{Pro}\label{pro2}
Let multi-indices $\gamma$ and $\kappa$ such that 
$$|\gamma|+|\kappa| \leq N+1, \hspace{10mm} |\gamma| \leq N \hspace{6mm} \text{and} \hspace{6mm} |\kappa| \leq N.$$
Then, for all $t \in [0,T[$,
$$\mathfrak{I}^3_{\gamma,\kappa}(t) \hspace{2mm} \lesssim \hspace{2mm} \frac{\epsilon^2}{ (1+t)^{1-\delta}}.$$
\end{Pro}
\begin{Pro}\label{pro3}
Let multi-indices $\gamma$ and $\kappa$ such that 
$$|\gamma|+|\kappa| \leq N+1, \hspace{10mm} |\gamma| \leq N \hspace{6mm} \text{and} \hspace{6mm} |\kappa| \leq N.$$
Then, for all $t \in [0,T[$,
$$\mathfrak{I}^2_{\gamma,\kappa}(t) \hspace{2mm} \lesssim \hspace{2mm} \frac{\epsilon^2}{(1+t)^{\frac{5}{4}}} .$$
\end{Pro}
Since $\E^5_N[f] (0) \leq \epsilon$, the combination of the energy inequality \eqref{energyfff}, Lemma \ref{reducing}, Corollary \ref{Corz} and Proposition \ref{pro1} (respectively \ref{pro2} and \ref{pro3}) imply that, for all $t \in [0,T[$ and if $\epsilon$ is small enough,
$$  \E^5_{N-3}[f](t) \hspace{1mm} \leq \hspace{1mm} 2\epsilon (1+t)^{\delta} \hspace{6mm} \text{(respectively} \hspace{2mm} \E^3_N[f](t) \hspace{1mm} \leq \hspace{1mm} 2\epsilon (1+t)^{\delta} \hspace{6mm} \text{and} \hspace{6mm} \E^2_N[f](t) \hspace{1mm} \leq \hspace{1mm} 2\epsilon \text{)}.$$
\subsection{Proof of Propositions \ref{pro1} and \ref{pro2}}

Let us fix $t \in [0,T[$ as well as multi-indices $\gamma$ and $\kappa$ satisfying 
$$|\gamma|+|\kappa| \leq N-2, \hspace{10mm} |\gamma| \leq N-3 \hspace{6mm} \text{and} \hspace{6mm} |\kappa| \leq N-3.$$
Using successively $z \lesssim 1+t+r$, the pointwise estimate on $\nabla_x Z^{\gamma} \phi$ given by Proposition \ref{estiphi} and then the bootstrap assumption \eqref{boot1}, we obtain
\begin{equation}\label{eq:followingthecomp}
 \mathfrak{I}^5_{\gamma, \kappa}(t) \hspace{2mm} \lesssim \hspace{2mm}  \left\|(1+r) \nabla_x Z^{\gamma} \phi \right\|_{L^{\infty}(\Sigma_t)} \left\| \int_{\R^3_v} z^5 |\ZZ^{\kappa} f| dv \right\|_{L^1(\Sigma_t)} \hspace{2mm} \lesssim \hspace{2mm} \frac{\epsilon}{1+t} \E^5_{N-3}[f](t)  \hspace{2mm} \lesssim \hspace{2mm} \frac{\epsilon^2}{ (1+t)^{1-\delta}},
\end{equation}
which concludes the proof of Proposition \ref{pro1}. We now turn on Proposition \ref{pro2} and we fix multi-indices $\gamma$ and $\kappa$ such that
$$|\gamma|+|\kappa| \leq N+1, \hspace{10mm} |\gamma| \leq N \hspace{6mm} \text{and} \hspace{6mm} |\kappa| \leq N.$$
\begin{itemize}
\item If $|\gamma| \leq N-3$, then, following the computations of \eqref{eq:followingthecomp} but using this time the bootstrap assumption \eqref{boot2}, we get
$$ \mathfrak{I}^3_{\gamma, \kappa}(t) \hspace{2mm} \lesssim \hspace{2mm} \frac{\epsilon}{1+t} \E^3_N[f](t) \hspace{2mm} \lesssim \hspace{2mm} \frac{\epsilon^2}{ (1+t)^{1-\delta}}.$$
\item Otherwise, $|\gamma| \geq N-2$ so $|\kappa| \leq 3 \leq N-6$. Hence, the pointwise decay estimate \eqref{decayf} and Proposition \ref{estiphi} gives us 
$$\forall \hspace{0.5mm} t \in [0,T[ , \hspace{1cm} \left\| \int_{\R^3_v} z^3 |\ZZ^{\kappa} f | dv \right\|_{L^{\infty}(\Sigma_t)} \hspace{2mm} \lesssim \hspace{2mm} \frac{\epsilon}{ (1+t)^{2-\delta}} \hspace{6mm} \text{and} \hspace{6mm} \left\| Z^{\gamma} \phi \right\|_{L^2(\Sigma_t)} \hspace{2mm} \lesssim \hspace{2mm} \frac{\epsilon}{\sqrt{1+t}}.$$
Thus, using the Cauchy-Schwarz inequality in $x$ and then in $v$, we obtain
\begin{eqnarray}
\nonumber \mathfrak{I}^3_{\gamma, \kappa}(t) & \lesssim &  \int_{\Sigma_t} \left| \nabla_x Z^{\gamma} \phi \right| \int_{\R^3_v} z^4 \left| \ZZ^{\kappa} f \right| dv dx  \\
 \nonumber & \lesssim &   \left\| \nabla_x Z^{\gamma} \phi \right\|_{L^2(\Sigma_t)} \left\| \int_{\R^3_v} z^4 \left| \ZZ^{\kappa} f \right| dv \right\|_{L^2(\Sigma_t)}  \\ \nonumber
& \lesssim &  \frac{\epsilon}{\sqrt{1+t}} \left\| \int_{\R^3_v} z^3 \left| \ZZ^{\kappa} f \right| dv \right\|^{\frac{1}{2}}_{L^{\infty}(\Sigma_t)} \left\| \int_{\R^3_v} z^5 \left| \ZZ^{\kappa} f \right| dv \right\|^{\frac{1}{2}}_{L^1(\Sigma_t)}  \\ 
& \lesssim &  \frac{\epsilon^2}{(1+t)^{\frac{3-\delta}{2}}} \sqrt{\E^5_{N-3}[f](s)} ds \hspace{2mm} \lesssim \hspace{2mm}  \frac{\epsilon^2}{(1+t)^{\frac{3}{2}-\delta}} \hspace{2mm} \lesssim \hspace{2mm} \frac{\epsilon^2}{(1+t)^{1-\delta}}. \label{eq:final}
\end{eqnarray}
\end{itemize}

\subsection{Proof of Proposition \ref{pro3}}

Let $t \in [0,T[$ and $\gamma$, $\kappa$ be multi-indices satisfying 
$$|\gamma|+|\kappa| \leq N+1, \hspace{10mm} |\gamma| \leq N \hspace{6mm} \text{and} \hspace{6mm} |\kappa| \leq N.$$
\begin{itemize}
\item If $|\gamma| \geq N-2$, then $\mathfrak{I}^2_{\gamma, \kappa}(t) \leq \mathfrak{I}^3_{\gamma, \kappa}(t) \lesssim \epsilon^2 (1+t)^{-\frac{5}{4}}$ by \eqref{eq:final}.
\item Otherwise $|\gamma| \leq N-3$ and we have, as $\tau_-^{-1} \leq z$ by Lemma \ref{zprop},
$$ \mathfrak{I}^2_{\gamma, \kappa}(t) \hspace{2mm} \lesssim \hspace{2mm} \int_{\Sigma_t} \int_{\R^3_v} \left( z\frac{r}{1+t+r}+\frac{z}{\tau_-} \right) \left| \nabla_x Z^{\gamma} \phi \right|  z^2 \left| \ZZ^{\kappa} f \right| dv dx .$$
Note now, according to Proposition \ref{estiphi}, that for all $(t,x) \in [0,T[ \times \R^3$,
$$\left( \frac{r}{1+t+r}+\tau_-^{-1} \right) \left|\nabla_x Z^{\gamma} \phi \right|(t,x) \hspace{2mm} \lesssim \hspace{2mm} \left( \frac{r}{1+t+r}+\frac{1}{1+|t-r|} \right) \frac{\epsilon }{(1+t)(1+r)} \hspace{2mm} \lesssim \hspace{2mm} \frac{\epsilon}{(1+t)^2}.$$
Consequently, using also the bootstrap assumption \eqref{boot2}, we finally get
$$ \mathfrak{I}^2_{\gamma, \kappa}(t) \hspace{2mm} \lesssim \hspace{2mm}  \frac{\epsilon}{(1+t)^2} \E^3_N[f](s) ds \hspace{2mm} \lesssim \hspace{2mm} \frac{\epsilon^2}{(1+t)^{2-\delta}} \hspace{2mm} \lesssim \hspace{2mm} \frac{\epsilon^2}{(1+t)^{\frac{5}{4}}}.$$
\end{itemize}

\section{$L^2$ estimate for the velocity average of the Vlasov field}\label{secL2}

The purpose of this section is to prove the following result.

\begin{Pro}\label{improbootL2}
There exists $\overline{C} >0$, a constant depending only on $N$, such that
$$ \forall \hspace{0.5mm} |\beta| \leq N, \hspace{3mm} \forall \hspace{0.5mm} t \in [0,T[, \hspace{1cm} \left\| \int_{\R^3_v} \left| \ZZ^{\beta} f \right| dv \right\|_{L^2(\Sigma_t)} \hspace{2mm} \leq \hspace{2mm} \overline{C} \frac{\epsilon}{1+t}.$$
\end{Pro}
This improves in particular the bootstrap assumption \eqref{bootL2} provided that $C_{L^2}$ is chosen larger than $\overline{C}+1$. Note first that, using the Cauchy-Schwarz-inequality in $v$ and then the pointwise decay estimate \eqref{decayf2} as well as the bootstrap assumption \eqref{boot3}, we have, for all $|\beta| \leq N-3$ and $t \in [0,T[$,
$$ \left\| \int_{\R^3_v} \left| \ZZ^{\beta} f \right| dv \right\|^2_{L^2(\Sigma_t)} \hspace{1mm} \lesssim \hspace{2mm}  \left\| \int_{\R^3_v} \left| \ZZ^{\beta} f \right| dv \right\|_{L^{\infty}(\Sigma_t)}  \left\| \int_{\R^3_v} \left| \ZZ^{\beta} f \right| dv \right\|_{L^1(\Sigma_t)} \hspace{1mm} \lesssim \hspace{2mm} \frac{\epsilon}{(1+t)^2} \E^2_N[f](t) \hspace{2mm} \lesssim \hspace{2mm} \frac{\epsilon^2}{(1+t)^2}.$$
It then remains to prove such an inequality for the higher order derivatives. For this, we proceed as in \cite{FJS} (subsection $4.5.7$) and we introduce
\begin{eqnarray}
\nonumber \mathcal{I} & := & \{ \beta \hspace{2mm} \text{multi-index} \hspace{1mm} / \hspace{1mm} N-5 \leq |\beta| \leq N \}  \hspace{2mm} = \hspace{2mm} \{ \beta^{1},...,\beta^{|I|} \},  \\ \nonumber
 \mathfrak{I} & := & \{ \xi \hspace{2mm} \text{multi-index} \hspace{1mm} / \hspace{1mm}  |\xi| \leq N-6 \} \hspace{2mm} = \hspace{2mm} \{ \xi^1,...,\xi^{|\I|} \} .
\end{eqnarray}
Let also $R$ and $W$ by two vector valued fields of respective length $|\mathcal{I}|$ and $|\I|$ such that
$$ R_i= \widehat{Z}^{\beta^i}f \hspace{10mm} \text{and} \hspace{10mm} W_i = \widehat{Z}^{\xi^i}f.$$
The goal now is to control the $L^2$ norm of $\int_{\R^3_v} |R|dv$. We put the derivatives of order $N-5$, $N-4$ and $N-3$ in $\mathcal{I}$ instead of $\mathcal{J}$ since our proof requires a pointwise decay estimate on $\int_{\R^3_v} z^3|W| dv$.
\begin{Def}
We call good coefficient any function $c \in C^{\infty}( \R^3_v , \R)$ such that
$$\forall \hspace{0.5mm} k \in \mathbb{N}, \hspace{2mm} \exists \hspace{0.5mm} C_k >0, \hspace{1cm} \left\| |v|^k \nabla_v^k c(v) \right\|_{L^{\infty}(\R^3_v)} \hspace{2mm} \leq \hspace{2mm} C_k.$$
\end{Def}
Note in particular that if $c$ is a good coefficient, such as $\frac{v^{\mu}}{v^0}$, and $\ZZ^{\beta} \in \VV^{|\beta|}$, then $\ZZ^{\beta} (c)$ is uniformly bounded in $v \in \R^3_v$. Let us now rewrite the Vlasov equations satisfied by the components of $R$ and $W$. We recall that the vector fields $Y_{pq}$ are defined by
$$Y_{pq} \hspace{2mm} = \hspace{2mm} \frac{v^p}{v^0} \partial_q -\frac{v^q}{v^0} \partial_p, \hspace{1cm} 1 \leq p < q \leq 3.$$

\begin{Pro}\label{L2bilan}
There exist three matrix-valued functions $A : [0,T[ \times \R^3_x \times \R^3_v \rightarrow \mathfrak M_{|\mathcal{I}|}$, $B : [0,T[ \times \R^3 \times \R^3_v  \rightarrow  \mathfrak M_{|\mathcal{I}|,|\I|}(\Vv)$ and $D : [0,T[ \times \R^3 \times \R^3_v \rightarrow  \mathfrak M_{|\I|}(\R)$ such that
$$\T_{\phi}(R)+AR \hspace{2mm} = \hspace{2mm} B W  \hspace{1.2cm} \text{and} \hspace{1.2cm} \T_{\phi}(W) \hspace{2mm} = \hspace{2mm} DW.$$
Moreover, if $i \in \llbracket 1, |\II| \rrbracket$,  $\T_{\phi}(R_i)$ is a linear combination of terms of the form
\begin{eqnarray}
\nonumber c^{\gamma}_j(v) \cdot \partial_l Z^{\gamma} \phi \cdot R_j, \hspace{4mm} & & \hspace{4mm} c^{\gamma , p,q}_j(v) \cdot x_n \cdot Y_{pq} Z^{\gamma} \phi \cdot R_j, \hspace{1.5cm} \text{with} \hspace{3mm} |\gamma| \leq N-6 \hspace{3mm} \text{and} \hspace{3mm} j \in \llbracket 1 , |\II| \rrbracket,   \\ 
\nonumber c^{\alpha}_k(v) \cdot \partial_l Z^{\alpha} \phi \cdot W_k, \hspace{4mm} & & \hspace{4mm} c^{\gamma , p,q}_k(v) \cdot x_n \cdot Y_{pq} Z^{\alpha} \phi \cdot W_k, \hspace{1.4cm} \text{with} \hspace{4.5mm} |\alpha| \leq N, \hspace{6mm} k \in \llbracket 1 , |\I| \rrbracket,
\end{eqnarray}
$(l,n,p,q) \in \llbracket 1,3 \rrbracket^4$ and where $c^{\gamma}_j$, $c^{\alpha}_k$, $c^{\gamma , p ,q}_j$, $c^{\alpha , p , q}_k$ are good coefficients. Similarly, the matrix $D$ is such that $\T_{\phi} ( W_k)$ can be bounded by a linear combination of terms of the form
$$ \left| \nabla_x Z^{\gamma} \phi \cdot W_k \right|, \hspace{10mm}  \frac{z r}{1+t+r} \left| \nabla_x Z^{\gamma} \phi \cdot W_k \right|, \hspace{1.2cm} \text{with} \hspace{6mm} |\gamma| \leq N \hspace{4mm} \text{and} \hspace{4mm} k \in \llbracket 1 , |\I| \rrbracket .$$
The matrix valued field $W$ satisfies the following pointwise decay estimates,
$$\forall \hspace{0.5mm} (t,x) \in [0,T[ \times \R^3, \hspace{1cm} \int_{\R^3_v} z^3|W| dv \hspace{2mm} \lesssim \hspace{2mm} \frac{\epsilon}{(1+t)^{2-\delta}} \hspace{1cm} \text{and} \hspace{1cm} \int_{\R^3_v} |W| dv \hspace{2mm} \lesssim \hspace{2mm} \frac{\epsilon}{(1+t)^2}.$$
\end{Pro}

\begin{proof}
Let $i \in \llbracket 1 , |\II| \rrbracket$ so that $R_i = \ZZ^{\beta^i} f$, where $|\beta^i| \geq N-5$. According to the commutation formula of Proposition \ref{comu}, $\T_{\phi}(R_i)$ can be written as a linear combination of terms of the form
$$ v^0 \nabla_x Z^{\gamma} \phi \cdot \nabla_v \ZZ^{\beta} f, \hspace{1cm} \text{with} \hspace{1cm} |\gamma| +|\beta| \leq N \hspace{6mm} \text{and} \hspace{6mm} |\beta| \leq N-1.$$
Expanding it in the orthonormal spherical frame in $v$ as in Lemma \ref{reducing}, we obtain terms such as
\begin{equation}\label{eq:bilanL2} \frac{v^m}{v^0} \partial_m Z^{\gamma} \phi \cdot v^m \partial_{v^m} \ZZ^{\beta} f, \hspace{1cm} c^{p,q}(v) \cdot Y_{pq} Z^{\gamma} \phi \cdot \widehat{\Omega}_{ab} \ZZ^{\beta} f, \hspace{1cm} c^{p,q}(v) \cdot x_n \cdot Y_{pq} Z^{\gamma} \phi \cdot \partial_l \ZZ^{\beta} f,
\end{equation}
where $c^{p,q}(v)$ is a good coefficient and $(n,l,a,b,p,q) \in \llbracket 1 , 3 \rrbracket^{6}$.
\begin{itemize}
\item If $|\beta| \geq N-6$, then
$$ \exists \hspace{0.5mm} j_1, \hspace{2mm} j_2, \hspace{2mm} j_3 \in \llbracket 1, |\II| \rrbracket, \hspace{11mm} R_{j_1} = v^m \partial_{v^m} \ZZ^{\beta} f, \hspace{6mm} R_{j_2} = \widehat{\Omega}_{ab} \ZZ^{\beta} f , \hspace{6mm} R_{j_3} = \partial_l \ZZ^{\beta} f$$ 
and $|\gamma| \leq N- |\beta| \leq 6 \leq N-6$.
\item Otherwise, $|\beta| \leq N-7$ and 
$$ \exists \hspace{0.5mm} k_1, \hspace{2mm} k_2, \hspace{2mm} k_3 \in \llbracket 1, |\I| \rrbracket, \hspace{11mm} W_{k_1} = v^m \partial_{v^m} \ZZ^{\beta} f, \hspace{6mm} W_{k_2} = \widehat{\Omega}_{ab} \ZZ^{\beta} f , \hspace{6mm} W_{k_3} = \partial_l \ZZ^{\beta} f.$$ 
\end{itemize}
Hence, all the terms of \eqref{eq:bilanL2} have the requested form. The construction of the matrix $D$ directly follows from the commutation formula of Proposition \ref{comu}, applied to $\ZZ^{\xi^k} f= W_k$, and Lemma \ref{reducing}. The pointwise decay estimates satisfied by $\int_{\R^3_v} |W| dv$ are given by \eqref{decayf} and \eqref{decayf2}.
\end{proof}

Let us now split $R$ into two parts $H+G$, where
\begin{eqnarray}
\nonumber T^{\chi}_{\phi}(H)+AH & = & 0, \hspace{2cm} H(0,.,.) \hspace{2mm} = \hspace{2mm} R(0,.,.), \\ \nonumber
T^{\chi}_{\phi}(G)+AG & = & BW, \hspace{1.5cm} W(0,.,.) \hspace{2mm} = \hspace{2mm} 0.
\end{eqnarray}
In order to obtain an $L^2$ estimate on the velocity average of $|R|$, we will prove
\begin{itemize}
\item a pointwise decay estimate on $\int_{\R^3_v} |H| dv$ by propagating a weighted $L^1_{x,v}$ norm of $H$. The computations will be similar to those of Section \ref{Secbootf} and we will in particular take advantage of the pointwise decay estimate that we have at our disposal on $Z^{\kappa} A$, for $|\kappa| \leq 3$. This will allow us to obtain an $L^2$ estimate on $\int_{\R^3_v} |H| dv$.
\item Since we do not control the derivatives of the matrix $B$, one cannot commute the transport equation satisfied by $G$ in order to derive pointwise estimates on it. For this, we introduce $K$, the solution of $\T^{\chi}_F(K)+AK+K D=B$ such that $K(0, \cdot , \cdot )=0$. Now, remark that
$$\T^{\chi}_{\phi}(KW)+AKW \hspace{2mm} = \hspace{2mm} BW \hspace{1cm} \text{and} \hspace{1cm} KW(0,\cdot , \cdot) \hspace{2mm} = \hspace{2mm} 0,$$ which implies that $G=KW$. Recall that $\int_{\R^3_v} |W| dv$ is a decaying function, so the inequality
\begin{eqnarray}
\nonumber \left\| \int_{\R^3_v} |G| dv \right\|^2_{L^2(\Sigma_t)} & = & \left\| \int_{\R^3_v} |KW| dv \right\|^2_{L^2(\Sigma_t)} \hspace{2mm} \lesssim \hspace{2mm} \left\| \int_{\R^3_v} |W| dv \int_{\R^3_v} |K|^2 |W| dv \right\|_{L^1(\Sigma_t)} \\  
& \lesssim & \left\| \int_{\R^3_v} |W| dv \right\|_{L^{\infty}(\Sigma_t)} \left\| \int_{\R^3_v} |K|^2 |W| dv \right\|_{L^1(\Sigma_t)}, \label{eq:inhoG}
\end{eqnarray}
will allow us to prove an estimate on the $L^2$ norm of the velocity average of $G$, provided that we control
\begin{equation}\label{defEG}
 \E_G(t) \hspace{2mm} := \hspace{2mm} \sum_{i=0}^{|\II|} \sum_{j=0}^{|\I|} \sum_{q=0}^{|\I|}  \int_{\Sigma_t} \int_{\R^3_v} \left| K_i^j \right|^2 \left| W_q \right| dv dx .
\end{equation}
\end{itemize}
Note that in view of the definition of $W$ and $R$, $G=KW$ and $R$ vanish for $|v| \leq 1$ according to Proposition \ref{vanishsmallv}. Since $H=R-G$, we deduce that $H$ also satisfies this property. We then obtain 
\begin{Pro}\label{supportGH}
The vector valued fields $H$ and $G$ both vanish for all $|v| \leq 1$, which implies that
$$T_{\phi}(H)+AH \hspace{2mm}= \hspace{2mm} 0, \hspace{2cm} H(0,.,.) \hspace{2mm} = \hspace{2mm} R(0,.,.).$$
\end{Pro}

\subsection{Study of the homogeneous system}

The purpose of this subsection is to obtain an $L^{\infty}$, and then an $L^2$, estimate on $\int_{\R^3_v} |H| dv$ through the use of Proposition \ref{decay}. We will then be led to commute three times the transport equation satisfied by $H$ and we then start by the following Lemma.

\begin{Lem}\label{comutaH}
Let $i \in \llbracket 1 , |\II| \rrbracket$, $\ZZ^{\beta} \in \VV^{|\beta|}$ and $\ZZ^{\zeta}  \in \VV^{|\zeta|}$,  with $|\beta|+|\zeta| \leq 3$. Then, $\ZZ^{\zeta} \left( T_{\phi} ( \ZZ^{\beta} H_i ) \right)$ can be bounded by a linear combination of
\begin{equation}\label{comuH}
 \left(1+ \frac{z r}{1+t+r} \right) \left| \nabla_x Z^{\gamma} \phi \right| | \ZZ^{\xi} H_j|, \hspace{10mm} \text{with} \hspace{10mm} |\gamma| \leq N-3, \hspace{5mm} |\xi| \leq 3 \hspace{5mm} \text{and} \hspace{5mm} j \in \llbracket 1, 3 \rrbracket.\end{equation}
\end{Lem}
\begin{proof}
Let $|\zeta|+|\beta| \leq 3$ and note first that
$$\ZZ^{\zeta} \left( T_{\phi} ( \ZZ^{\beta} H_i ) \right) \hspace{2mm} = \hspace{2mm} - [T_{\phi} , \ZZ^{\zeta}  ]( \ZZ^{\beta} H_i )+ [T_{\phi} ,\ZZ^{\zeta} \ZZ^{\beta} ](H_i)- \ZZ^{\zeta} \ZZ^{\beta} \left( T_{\phi} ( H_i ) \right).$$
Applying the commutation formula of Proposition \ref{comu} and then Lemma \ref{reducing}, we directly obtain that we can bound $-[T_{\phi} , \ZZ^{\zeta} ]( \ZZ^{\beta} H_i)+ [T_{\phi} ,\ZZ^{\zeta} \ZZ^{\beta} ](H_i)$ by a linear combination of terms given in \eqref{comuH}. By Proposition \ref{L2bilan} and the fact that if $c(v)$ is a good coefficient, then $\ZZ^{\kappa} (c) \in L^{\infty}(\R^3_v)$, we can estimate $\ZZ^{\zeta}  \ZZ^{\beta} (T_{\phi} (H_i))$ by a linear combination of
$$ \left| \ZZ^{\xi} \left( \nabla_x Z^{\gamma} \phi \cdot H_j \right) \right| \hspace{5mm} \text{and} \hspace{5mm}  \left| \ZZ^{\xi} \left(x_n \cdot Y_{pq} Z^{\gamma} \phi \cdot H_j \right) \right|, \hspace{1cm} \text{with} \hspace{5mm} j \in \llbracket 1, |\II| \rrbracket, \hspace{3mm} |\gamma| \leq N-6, \hspace{3mm} |\xi| \leq 3$$
and $(n,p,q) \in \llbracket 1,3 \rrbracket^3$. Note now that
$$ [ \partial_{\mu} , Y_{pq} ] \hspace{2mm} = \hspace{2mm} [ v^i \partial_{v^i} , Y_{pq} ] \hspace{2mm} = \hspace{2mm} 0, \hspace{1cm} [ S , Y_{pq} ] \hspace{2mm} = \hspace{2mm} -Y_{pq}, \hspace{1cm} [ \widehat{\Omega}_{ij} , Y_{pq} ] \hspace{2mm} = \hspace{2mm} \delta^i_p Y_{qj}+\delta^i_q Y_{jp}+\delta^j_p Y_{iq}+\delta^j_q Y_{pi}.$$
Using also that, for $\ZZ^{\alpha} \in \VV^{|\alpha|}$, $\ZZ^{\alpha} (x_n) \in \{0, \pm x_1, \pm x_2, \pm x_3 \}$ and
$$ \left| \ZZ^{\alpha} \nabla_x Z^{\gamma} \phi \right| \hspace{2mm} \lesssim \hspace{2mm} \sum_{|\kappa| \leq |\alpha|+|\gamma|}  \left| \nabla_x Z^{\kappa} \phi \right|,$$
we obtain, using Leibniz formula, that $\ZZ^{\beta} (T_{\phi} (H_i))$ can be bounded by a linear combination of
$$ \left|  \nabla_x Z^{\kappa} \phi \cdot \ZZ^{\theta} H_j  \right| \hspace{10mm} \text{and} \hspace{10mm}  r\left| Y_{pq} Z^{\kappa} \phi \cdot \ZZ^{\theta} H_j  \right|, \hspace{1cm} \text{with} \hspace{10mm} |\kappa| \leq N-3 \hspace{5mm} \text{and} \hspace{5mm} |\theta| \leq 3.$$
It then only remains to use \eqref{eq:number1} in order to estimate $r\left| Y_{pq} Z^{\kappa} \phi \cdot \ZZ^{\theta} H_j  \right|$ by the terms given in \eqref{comuH}.
\end{proof}
Let us now introduce the following energy norms
$$ \E^3_H(t) \hspace{2mm} := \hspace{2mm} \sum_{i=1}^{|\II|} \E^3_3[H_i](t) \hspace{1cm} \text{and} \hspace{1cm} \E^2_H(t) \hspace{2mm} := \hspace{2mm} \sum_{i=1}^{|\II|} \E^2_3[H_i](t).$$
\begin{Pro}\label{HPro}
If $\epsilon$ is small enough, we have $\E^3_H(t) \leq 2C_H\epsilon (1+t)^{\delta}$ and $\E^2_H(t) \leq 2C_H \epsilon$ for all $t \in [0,T[$. It implies
$$ \forall \hspace{0.5mm} t \in [0,T[, \hspace{1cm} \left\| \int_{\R^3_v} |H| dv \right\|_{L^2(\Sigma_t)} \hspace{2mm} \lesssim \hspace{2mm} \frac{\epsilon}{1+t}.$$
\end{Pro}
\begin{proof}
Applying Lemma \ref{comutaH} and using that $H(0,.,.)=R(0,.,.)$ as well as the smallness hypotheses on $\int_{\Sigma_0} \int_{\R^3_v} (1+|x|)^3 |\ZZ^{\beta} f| dv dx$ for all $|\beta| \leq N+3$, we directly obtain, for $\epsilon$ enough, that
$$ \exists \hspace{0.5mm} C_H >0, \hspace{1cm} \E^2_H(0) \hspace{2mm} \leq \hspace{2mm} \E^3_H(0) \hspace{2mm} \leq \hspace{2mm} C_H \epsilon.$$
The bounds on $\E^3_H$ and $\E^2_H$ then follow from a bootstrap argument and the use of the energy inequality of Proposition \ref{energyparti}. Indeed, according to Lemma \ref{comutaH}, the potential is differentiated at most $N-3$ times in all the terms composing $\ZZ^{\zeta} \left( T_{\phi}(\ZZ^{\beta} H_i) \right)$, when $|\zeta|+|\beta| \leq 3$, and can then be estimated pointwise (see Proposition \ref{estiphi}). Hence, for $\E^3_H$ (respectively $\E^2_H$), one just has to follow the computations of the proof of Proposition \ref{pro1} (respectively \ref{pro3} for the cases $|\gamma| \leq N-3$). Now, according to Proposition \ref{decay}, we have
$$ \forall \hspace{0.5mm} t \in [0,T[ , \hspace{1cm} \left\| \int_{\R^3_v} |H| dv \right\|_{L^{\infty}(\Sigma_t)} \hspace{2mm} \lesssim \hspace{2mm} \frac{\E_H^2(t)}{(1+t)^2} \hspace{2mm} \lesssim \hspace{2mm} \frac{\epsilon}{(1+t)^2}.$$
We finally obtain, using the Cauchy-Schwarz inequality in $v$, for all $t \in [0,T[$,
$$ \left\| \int_{\R^3_v} |H| dv \right\|^2_{L^2(\Sigma_t)} \hspace{2mm} \lesssim \hspace{2mm} \left\| \int_{\R^3_v} |H| dv \int_{\R^3_v} |H| dv\right\|_{L^1(\Sigma_t)} \hspace{2mm} \lesssim \hspace{2mm} \left\| \int_{\R^3_v} |H| dv \right\|_{L^{\infty}(\Sigma_t)} \E^2_H(t) \hspace{2mm} \lesssim \hspace{2mm} \frac{\epsilon^2}{(1+t)^2}.$$
\end{proof}
\subsection{Study of the inhomogeneous system}

In order to prove that $\E_G$ is uniformly bounded in time, we will need to control sufficiently well
$$\E_G^1(t) \hspace{2mm} := \hspace{2mm} \sum_{i=0}^{|\II|} \sum_{j=0}^{|\I|} \sum_{q=0}^{|\I|}  \int_{\Sigma_t} \int_{\R^3_v} z\left| K_i^j \right|^2 \left| W_q \right| dv dx $$
and to apply Proposition \ref{energyparti}. For this, remark that
\begin{equation}\label{sourceinhophi}
\T_{\phi} \left( |K^j_i|^2 W_q \right) \hspace{2mm} = \hspace{2mm} \T^{\chi}_{\phi} \left( |K^j_i|^2 W_q \right) \hspace{2mm} = \hspace{2mm} |K^j_i |^2 D^r_q W_r-2\left(A^r_i K^j_r +K^r_i D^j_r \right) K^j_i W_q+2 B^j_iK^j_iW_q.
\end{equation}

\begin{Pro}\label{GPro}
We have, if $\epsilon$ is small enough,
$$ \forall \hspace{0.5mm} t \in [0,T[, \hspace{1cm} \E_G^1(t) \leq \epsilon (1+t)^{\delta} \hspace{1cm} \text{and} \hspace{1cm} \E_G(t) \leq \epsilon.$$
\end{Pro}
\begin{proof}
Since $\E_G^1(0)= \E_G(0)=0$, there exists $\widetilde{T} \in ]0,T]$ such that 
$$ \forall \hspace{0.5mm} t \in [0,\widetilde{T}[, \hspace{1cm} \E_G^1(t) \leq 2\epsilon (1+t)^{\delta} \hspace{1cm} \text{and} \hspace{1cm} \E_G(t) \leq 2\epsilon.$$
In order to improve these boostsrap assumptions, combine the energy inequality of Proposition \ref{energyparti} with \eqref{sourceinhophi} to get, for $t \in [0,\widetilde{T}[$, that
\begin{eqnarray}
\nonumber \E_G(t) \hspace{-2.2mm} & \leq & \hspace{-2.2mm} \sum_{\begin{subarray}{} 1 \leq i \leq |\II| \\ 1 \leq j,q \leq |\I| \end{subarray}} \hspace{-2.5mm}\int_0^t \hspace{-0.5mm} \int_{\Sigma_s} \hspace{-0.5mm} \int_{\R^3_v} \left| \T_{\phi} \left( |K^j_i|^2 W_q\right) \right| \frac{dv}{v^0} dx ds \hspace{0.6mm} \lesssim \hspace{0.6mm} \int_0^t \hspace{-0.5mm} \int_{\Sigma_s} \hspace{-0.5mm} \int_{\R^3_v} \left( (|A|+|D|)|K|^2|W| + |B||K||W| \right)\frac{dv}{v^0} dx ds \\ \nonumber
\E^1_G(t)  \hspace{-2.2mm} & \leq & \hspace{-2.4mm} \int_0^t \hspace{-0.6mm} \int_{\Sigma_s} \hspace{-0.6mm} \int_{\R^3_v} \hspace{-0.6mm} z\left( (|A|+|D|)|K|^2|W| + |B||K||W| \right)\frac{dv}{v^0} dx ds \hspace{-0.2mm} + \hspace{-0.8mm} \int_0^t \hspace{-0.6mm} \int_{\Sigma_s} \hspace{-0.6mm} \int_{\R^3_v} \hspace{-0.6mm} v^0 \left| \nabla_x \phi \cdot \nabla_v (z) \right| |K|^2 |W|\frac{dv}{v^0} dx ds.
\end{eqnarray}
Using first Proposition \ref{L2bilan} as well as \eqref{eq:number1} and then the pointwise decay estimates of Proposition \ref{estiphi} as well as $\tau_- \leq z$ (see Lemma \ref{zprop}), we have
$$ |A|+|D| \hspace{1.2mm} \lesssim \hspace{1.2mm} \left( 1+z\frac{r}{1+t+r}  \right) \sum_{|\gamma| \leq N-3}  |\nabla_x Z^{\gamma} \phi | \hspace{1.2mm} \lesssim \hspace{1.2mm} z\left( \frac{1}{1+|t-r|}+\frac{r}{1+t+r}  \right) \frac{\epsilon}{(1+r)(1+t)} \hspace{1.2mm} \lesssim \hspace{1.2mm} \frac{\epsilon z}{(1+t)^2}.$$
Similarly, one has, since $z \lesssim 1+t+r$,
$$ |A|+|D| \hspace{2mm} \lesssim \hspace{2mm} \left( 1+z\frac{r}{1+t+r}  \right) \sum_{|\gamma| \leq N-3}  |\nabla_x Z^{\gamma} \phi | \hspace{2mm} \lesssim \hspace{2mm} \frac{\epsilon}{1+t}.$$
Hence, as $v^0 \geq 1$ on the support of $W$,
\begin{eqnarray}
\nonumber \int_0^t \int_{\Sigma_s} \int_{\R^3_v} (|A|+|D|)|K|^2|W| \frac{dv}{v^0} dx ds \hspace{-0.4mm} & \lesssim & \hspace{-0.4mm} \int_0^t \hspace{-0.2mm} \frac{\epsilon}{(1+s)^2} \int_{\Sigma_s} \hspace{-0.2mm} \int_{\R^3_v} z|K|^2|W| dv dx ds \hspace{1.4mm} \lesssim \hspace{1.4mm} \int_0^t \frac{\epsilon \hspace{0.2mm} \E^1_G(s)}{(1+s)^2}  ds \hspace{1.4mm} \lesssim \hspace{1.4mm} \epsilon^2, \\ 
\nonumber \int_0^t \int_{\Sigma_s} \int_{\R^3_v} z(|A|+|D|)|K|^2|W| \frac{dv}{v^0} dx ds\hspace{-0.5mm} & \lesssim & \hspace{-0.4mm}  \int_0^t \frac{\epsilon \hspace{0.2mm} \E^1_G(s)}{1+s}  ds \hspace{2mm} \lesssim \hspace{2mm} \epsilon^2 (1+t)^{\delta}.
\end{eqnarray}
\end{proof}
Applying Lemma \ref{reducing} with $|\gamma|=0$ and $g=z$ and then Lemma \ref{zprop}, we have
$$v^0 \left| \nabla_x \phi \cdot \nabla_v (z) \right| \hspace{2mm} \lesssim \hspace{2mm} \left\| \nabla_x \phi \right\|_{L^{\infty}(\Sigma_t)} ( z |\nabla_x z|+\sum_{ \ZZ \in \VV} |\ZZ(z)| ) \hspace{2mm} \lesssim \hspace{2mm} z \left\| \nabla_x \phi \right\|_{L^{\infty}(\Sigma_t)}.$$
According to the pointwise decay estimate of Proposition \ref{estiphi} and that $v^0 \geq 1$ on the support of $W$, we then obtain
\begin{eqnarray}
\nonumber \int_0^t \int_{\Sigma_s} \int_{\R^3_v} v^0 \left| \nabla_x \phi \cdot \nabla_v (z) \right| |K|^2 |W|\frac{dv}{v^0} dx ds & \lesssim & \int_0^t \left\| \nabla_x \phi \right\|_{L^{\infty}(\Sigma_s)} \int_{\Sigma_s} \int_{\R^3_v} z |K|^2 |W|\frac{dv}{v^0} dx ds \\ \nonumber 
& \lesssim & \int_0^t \frac{\epsilon \hspace{0.5mm} \E_G^1(s)}{(1+s)^2} ds \hspace{2mm} \lesssim \hspace{2mm} \epsilon^2.
\end{eqnarray}
Note now, using Proposition \ref{L2bilan} and \eqref{eq:number1}, that
$$|B| \hspace{2mm} \lesssim \hspace{2mm} \left( 1+z\frac{r}{1+t+r}  \right) \sum_{|\alpha| \leq N}  |\nabla_x Z^{\alpha} \phi | \hspace{2mm} \lesssim \hspace{2mm} z \sum_{|\alpha| \leq N}  |\nabla_x Z^{\alpha} \phi |.$$ 
We then obtain, using the Cauchy-Schwarz inequality in $x$ and that $v^0 \geq 1$ on the support of $W$,
$$\int_0^t \int_{\Sigma_s} \int_{\R^3_v} z|B||K||W| \frac{dv}{v^0} dx ds \hspace{2mm} \lesssim \hspace{2mm} \sum_{|\alpha| \leq N} \int_0^t \| \nabla_x Z^{\alpha} \phi \|_{L^2(\Sigma_s)} \left\| \int_{\R^3_v} z^2 |K||W| dv \right\|_{L^2(\Sigma_s)} ds. $$
The Cauchy-Schwarz inequality in $v$ and the pointwise decay estimate $\int_{\R^3_v} z^3 |W|dv \lesssim \epsilon (1+t)^{-2+\delta}$ give
$$\left\| \int_{\R^3_v} z^2 |K||W| dv \right\|_{L^2(\Sigma_s)} \hspace{1mm} \lesssim \hspace{2mm} \left\| \int_{\R^3_v} z^3 |W|dv \int_{\R^3_v} z |K|^2|W| dv \right\|^{\frac{1}{2}}_{L^1(\Sigma_s)} \hspace{2mm} \lesssim \hspace{2mm} \frac{\sqrt{\epsilon}}{(1+s)^{1-\frac{\delta}{2}}} \sqrt{\E^1_G(s)} \hspace{2mm} \lesssim \hspace{2mm} \frac{\epsilon}{(1+s)^{1-\delta}} .$$
Finally, by Proposition \ref{estiphi}, we finally obtain
$$\int_0^t \int_{\Sigma_s} \int_{\R^3_v} z|B||K||W| \frac{dv}{v^0} dx ds \hspace{2mm} \lesssim \hspace{2mm} \sum_{|\alpha| \leq N} \int_0^t \| \nabla_x Z^{\alpha} \phi \|_{L^2(\Sigma_s)} \frac{\epsilon^2}{(1+s)^{1-\delta}} ds \hspace{2mm} \lesssim \hspace{2mm} \int_0^t \frac{\epsilon^2}{(1+s)^{\frac{3}{2}-\delta}} ds \hspace{2mm} \lesssim \hspace{2mm} \epsilon^2,$$
which concludes the improvement of the bootstrap assumptions on $\E_G$ and $\E^1_G$.

\subsection{End of the proof of Proposition \ref{improbootL2}}

Let $\ZZ^{\beta} \in \VV^{|\beta|}$ such that $N-2 \leq |\beta| \leq N$ (recall that we already treated the lower order derivatives). Then, there exists $i \in \llbracket 1, |\mathcal{I}|  \rrbracket$ such that $H_i+G_i = R_i = \ZZ^{\beta} f$. According to \eqref{eq:inhoG}, we have
$$\forall \hspace{0.5mm} t \in [0,T[, \hspace{1cm} \left\| \int_{\R^3_v} |\ZZ^{\beta} f | dv \right\|^2_{L^2(\Sigma_t)} \hspace{2mm} \leq \hspace{2mm} \left\| \int_{\R^3_v} |H_i | dv \right\|^2_{L^2(\Sigma_t)}+\left\| \int_{\R^3_v} |W | dv \right\|_{L^{\infty}(\Sigma_t)} \E_G(t).$$
The result then follows from Propositions \ref{HPro} and \ref{GPro} as well as $\int_{\R^3_v} |W | dv \lesssim \epsilon (1+t)^{-2}$ (see Proposition \ref{L2bilan}).

\section{Acknowledgments}

Part of this work was funded by the European Research Council under the European Union's Horizon 2020 research and innovation program (project GEOWAKI, grant agreement 714408).

\renewcommand{\refname}{References}
\bibliographystyle{abbrv}
\bibliography{biblio}

\end{document}